\documentclass[reqno,final]{amsart}
\usepackage{natbib}  %Nature-like bibliography
\usepackage{fancyhdr} %Headers
\usepackage{color} %Color definition
\usepackage{hyperref} %Internal and external links
\usepackage{graphicx} %Graphics inclusion

\usepackage{pstricks}
\usepackage{amssymb}
\usepackage{t1enc}
\usepackage[latin1]{inputenc}
\usepackage{amsmath,latexsym,amssymb,graphicx,dsfont,amsthm,amsfonts}
\usepackage{color,enumitem}
\usepackage{tikz}
\usepackage{hyperref}
\usepackage{umoline,comment}

\definecolor{aleacolor}{rgb}{0.16,0.59,0.78}

\hypersetup{
breaklinks,
colorlinks=true,
linkcolor=aleacolor,
urlcolor=aleacolor,
citecolor=aleacolor}

\renewcommand{\headheight}{11pt}
\renewcommand{\cite}{\citet}

\theoremstyle{plain}
\newtheorem{theorem}{Theorem}[section]                                          
\newtheorem{proposition}[theorem]{Proposition}                          
\newtheorem{lemma}[theorem]{Lemma}
\newtheorem{corollary}[theorem]{Corollary}

\theoremstyle{definition}

\theoremstyle{remark}
\newtheorem{remark}[theorem]{Remark}
\newtheorem{example}[theorem]{Example}

\makeatletter \@addtoreset{equation}{section} \makeatother
\renewcommand\theequation{\thesection.\arabic{equation}}
\renewcommand\thefigure{\thesection.\arabic{figure}}
\renewcommand\thetable{\thesection.\arabic{table}}

\newcommand{\aleaIndex}[1]{\href{http://alea.impa.br/english/index_v#1.htm}{\bf #1}}
\eheader{}{\aleaIndex{}}{2020}{1}{37}

\newcommand{\E}{\mathbb{E}}

\renewcommand{\P}{\mathbb{P}}

\newcommand{\N}{\mathbb{N}}
\newcommand{\Z}{\mathbb{Z}}

\newcommand{\W}{\mathbf{W}}

\newcommand{\e}[1]{\mathbf{e}_{#1}}
\renewcommand{\i}[1]{\mathbf{i}_{#1}}

\begin{document}

\title[Asymptotic Word Length of Random Walks on HNN extensions]{Asymptotic Word Length of Random Walks on \\ HNN Extensions}
%\author{Lorenz A. Gilch\footnotemark[1]}

\author{Lorenz A. Gilch}

%\footnote{Universit\'e de Gen\`eve; email: Lorenz.Gilch@unige.ch; supported by German Research Foundation (DFG) grant
%GI 746/1-1}

\address{University of Passau\newline
Chair of Stochastics and its Applications\newline 
Innstr. 33, \newline
94032 Passau, Germany}

\email{Lorenz.Gilch@uni-passau.de}
\urladdr{http://www.math.tugraz.at/$\sim$gilch/}
%\date{\today}

%\maketitle

%\centerline{\scshape Lorenz A. Gilch}
%\medskip
%{\footnotesize
% \centerline{Graz University of Technology, Graz, Austria}}

%\let\oldthefootnote\thefootnote
%\renewcommand{\thefootnote}{\fnsymbol{footnote}}
%\footnotetext[1]{Supported by German Research Foundation (DFG) grant GI 746/1-1.}
%\let\thefootnote\oldthefootnote

\begin{abstract}
In this article we consider transient random walks on HNN extensions of finitely generated groups. We prove that the rate of escape w.r.t. some generalised word length exists. Moreover, a central limit theorem with respect to the generalised word length is derived. Finally, we show that the rate of escape, which can be regarded as a function in the finitely many parameters which describe the random walk, behaves as a real-analytic function in terms of probability measures of constant support. 
\end{abstract}
%{\scshape Keywords:} Random Walks, Free Products by Amalgamation, Entropy.
%\newline {\scshape AMS 2000 Mathematics Subject Classification: } Primary ; Secondary .

\maketitle

\section{Introduction}
\label{sec:introduction} 

Consider a finitely generated group $G_0$, which contains two isomorphic, finite subgroups $A,B$ with isomorphism $\varphi: A\to B$. Let $S_0\subseteq G_0$ be a finite set which generates $G_0$ as a semigroup, and let $t$ be an additional symbol/letter not contained in $G_0$.
The \textit{HNN extension of $G_0$ with respect to $(A,B,\varphi)$} is given by the set $G$ of all finite words over the alphabet $G_0\cup\{t,t^{-1}\}$, where two words  $w_1,w_2\in G$ are identified as the same element of G if one can transform $w_1$ to $w_2$ by applying the relations inherited from $G_0$ or applying one of the following rules:
$$
\forall a\in A: at=t\varphi(a)\ \textrm{ and } \ \forall b\in B: bt^{-1}=t^{-1}\varphi^{-1}(b). 
$$
A natural group operation on $G$ is given by concatenation of words with possible cancellations of letters in the middle; the empty word $e$ is the group identity. This group construction was introduced by Higman, Neumann and Neumann (see \cite{hnn:49}), whose initials lead to the abbreviation HNN. As we will see later, we can write each $g\in G$ in a unique normal form over some alphabet $\mathcal{A}\subset G_0\cup\{t,t^{-1}\}$. We denote by $\Vert g\Vert$ the word length of $g\in G$ over the \mbox{alphabet $\mathcal{A}$.}
\par
Consider now a group-invariant, transient random walk $(X_n)_{n\in\N_0}$ on $G$ governed by  probability measure $\mu$ with  $\mathrm{supp}(\mu)=S_0\cup\{t,t^{-1}\}$. One important random walk invariant is the \textit{rate of escape w.r.t. the word length} given by the almost sure constant limit $\mathfrak{l}=\lim_{n\to\infty}\Vert X_n\Vert/n$, which exists due to Kingman's subadditive ergodic theorem (see \cite{kingman}). The starting point of this article was the question whether $\mathfrak{l}$ -- regarded as a function in the finitely many parameters $\mu(g)$, \mbox{$g\in S_0\cup\{t,t^{-1}\}$} -- varies real-analytically  in terms of probability measures of constant support. We will study this question in a more generalised setting. For this purpose, let the function \mbox{$\ell: G_0\cup \{t,t^{-1}\}\to[0,\infty)$} represent a ``word length/weight''. We can naturally extend $\ell$ to a length function on $G$ as follows: if $g=g_1\dots g_n\in G$ has the above mentioned normal form representation over the alphabet $\mathcal{A}$ then we set
$$
\ell(g)=\ell(g_1\ldots g_n):=\sum_{i=1}^n \ell(g_i), 
$$
The \textit{asymptotic word length w.r.t. the length function $\ell$} is given by 
$$
\lambda_\ell=\lim_{n\to\infty}\ell(X_n)/n,
$$ 
provided the limit exists. We will also call $\lambda_\ell$ the \textit{rate of escape} or \textit{drift w.r.t. $\ell$}. For arbitrary length functions $\ell$, existence of the rate of escape w.r.t. $\ell$ is not guaranteed a-priori and can \textit{not} be deduced from Kingman's subadditive ergodic theorem in general; see Remark \ref{rem:remarks}.
This article addresses to typical, related questions like existence of the rate of escape  $\lambda_\ell$ (including formulas), a central limit theorem for $\lambda_\ell$ and its real-analytic behaviour in terms of probability measures of constant support. In the following let me explain the importance of these questions for random walks on HNN extensions from three different points of view, namely from the view of random walks on regular languages, from the view of group theory and from the view of analyticity of random walk invariants.
\par 
Due to the unique representation of each $g\in G$ over the (possibly infinite) alphabet $\mathcal{A}$ we may consider $(X_n)_{n\in\N_0}$ as a random walk on a regular language, where at each instant of time only a bounded number of letters at the end of the current word may be modified, removed or added. This class of random walks have been studied in large variety, but mostly for regular languages over \textit{finite} alphabets. Amongst others, 
\mbox{\cite{malyshev,malyshev-inria},} \cite{malyshev2}, %Yambartsev and Zamyatin \cite{yambartsev-zamyatin} 
and  \cite{lalley} investigated random walks on regular languages over finite alphabets.
In particular, \mbox{Malyshev} proved limit theorems concerning existence
of the stationary distribution and the rate of escape w.r.t. the word length. \cite{gilch:08} proved existence of the rate of escape w.r.t. general length functions for random walks on regular languages. 
All the articles above study   regular languages generated by \textit{finite} alphabets. 
Straight-forward adaptions of the proofs concerning the questions under consideration in the present article are not possible. This article extends results concerning existence of the drift from the finite case to the \textit{infinite} case in the setting of HNN extensions.
Studying the rate of escape w.r.t. $\ell$ deserves its own right, since the transient random walks studied in this article converge almost surely to some infinite random word $\omega$ over the alphabet $\mathcal{A}$ in the sense that the length of the common prefix  of  $X_n$ and $\omega$ increases as $n\to\infty$. As an application from information theory one may, e.g., consider $X_n$ as the state of a stack (a last-in first-out queue used in many fundamental algorithms of computer science) at time $n$, and each stabilised letter at the beginning of $X_n$ produces some final ``cost''. Hence, the rate of escape w.r.t. $\ell$  describes the average asymptotic cost.
\par
Let me now outline the importance of the questions under consideration from a group theoretical point of view.
The importance of HNN extensions is due to \textit{Stallings' Splitting Theorem} (see  \cite{stallings:71}): a finitely generated group $\Gamma$ has more than one (geometric) end if and only if $\Gamma$ admits a non-trivial decomposition as a free product by amalgamation or an HNN extension over a finite subgroup. Let me summarize some results about random walks on free products, which are amalgams over the trivial subgroup. For free products of finite groups,  \cite{mairesse1} computed an explicit formula for the rate of escape and the asymptotic entropy by solving a finite system of polynomial ``traffic equations''. In \cite{gilch:11}  different formulas for the rate of escape with respect to the word length of random walks on free products of graphs by three different techniques were computed. The main tool in \cite{gilch:11} was a heavy use of generating function techniques, which will also play a crucial role in the present article. Asymptotic behaviour of return probabilities of random walks on free products has also been studied in many ways; e.g., see \cite{gerl-woess},  \cite{woess3},  \cite{sawyer78},  \cite{cartwright-soardi},  \cite{lalley93}, and  \cite{candellero-gilch}. Random walks on amalgams have been studied in \cite{cartwright-soardi} and \cite{gilch:08}, where a formula for the rate of escape has been established for amalgams of finite groups. 
While random walks on free products have been studied in many ways due to their tree-like structure and random walks on amalgams at least to some extent, random walks on HNN extensions, in general,  have experienced much less attention.  \cite{woess-isr} proved that irreducible random walks with finite range on HNN extensions converge almost surely to infinite words over the alphabet $\mathcal{A}$ and that the set of infinite words together with the hitting distribution form the Poisson boundary.  Further valuable contributions have been done by  \cite{kaimanovich:91} and by  \cite{cuno-sava-huss}, who  studied the Poisson-Furstenberg boundary of random walks  on Baumslag-Solitar groups, which  form a special class of HNN extensions. The present article shall encourage further study of random walks on HNN extensions.
%Let me now outline the importance of the questions under consideration from a group theoretical point of view.
%It is well-known that the rate of escape w.r.t. the word length and also the rate of escape w.r.t. the natural graph metric of the underlying Cayley graph exist  for  random walks on groups, which are governed by finitely supported probability measures. This follows directly from Kingman's subadditive ergodic theorem; see Kingman \cite{kingman}, Derriennic \cite{derriennic} and Guivarc'h \cite{guivarch}. For arbitrary length functions $\ell$, existence of the rate of escape w.r.t. $\ell$ is not guaranteed a-priori. We remark that, in general, the rate of escape w.r.t. the natural graph metric of the Cayley graph of $G$ w.r.t. the generating set $S_0\cup\{t,t^{-1}\}$ can \textit{not} necessarily be described via a length function using \textit{stabilising} normal forms of elements of $G$; this is due to the quite subtle behaviour of shortest paths in the Cayley graph, which needs a different approach and goes beyond of the scope of this article; for a discussion on these problems, see Remarks \ref{rem:remarks} and \ref{rem:remarks2}.
\par
Another main goal of this article is to derive a central limit theorem related to the rate of escape $\lambda_\ell$. If $(Z_n)_{n\in\N_0}$ is a random walk on $\Z^d$ satisfying some second moment condition, then the classical central limit theorem states that \mbox{$(Z_n-n\cdot v)/n^{1/2}$} converges in distribution to $\mathcal{N}(0,\sigma^2)$, where $v$ is the rate of escape w.r.t. the natural distance on the lattice and $\sigma^2$ is the asymptotic variance. A natural question going back to  \cite{bellman} and \cite{furstenberg-kesten60} is whether this law can be generalized to random walks on finitely generated groups w.r.t. some word metrics. However, a central limit theorem can not be stated in the general setting:  \cite{bjoerklund10} used results of  \cite{erschler99} and \cite{erschler01} to construct a counterexample. Nonetheless, in several situations central limit theorems have been established; e.g.,  \cite{sawyer},  \cite{lalley93} and  \cite{ledrappier01} proved central limit theorems for free groups,  \cite{woess2} for trees with finitely many cone types, and \cite{bjoerklund10} for hyperbolic groups with respect to the Greenian metric.
\par 
The third main goal of this article will be to show that $\lambda_\ell$ varies real-analytically in terms of probability measures of constant support. The question of analyticity goes back to Kaimanovich and Erschler who asked whether drift and entropy of random walks on groups vary continuously (or even analytically) when the support of single step transitions is kept constantly; for counterexamples, see \mbox{Remark \ref{rem:counterexample-analyticity}.}
This question has been studied in great variety, amongst others, by  \cite{ledrappier12,ledrappier12-2},  \cite{mathieu:15} and  \cite{gilch:07,gilch:11,gilch:15}.  \cite{haissinsky-mathieu-mueller12} proved analyticity of the drift for random walks on surface groups and also established a central limit theorem for the word length. The  survey article of \cite{gilch-ledrappier13} collects several results on analyticity of drift and entropy of random walks on groups. Last but not least, the excellent work of  \cite{gouezel:15} shows that the rate of escape w.r.t. some word distance, the asymptotic variance and the asymptotic entropy vary real-analytically for random walks on hyperbolic groups. However,  HNN extensions do not necessarily have to be hyperbolic, which makes it interesting to study the question of analyticity of the rate of escape w.r.t. the word length for random walks on HNN extensions.
\par
Finally, let me mention that another random walk's speed invariant is given by the rate of escape w.r.t. the natural graph metric of the underlying Cayley graph of $G$ w.r.t. the generating set $S_0\cup\{t,t^{-1}\}$, which exists due to Kingman's subadditive ergodic theorem; see \cite{kingman},  \cite{derriennic} and  \cite{guivarch}. We remark that, in general, the rate of escape w.r.t. the natural graph metric can \textit{not} necessarily be described via a length function using \textit{stabilising} normal forms of elements of $G$; this is due to the quite subtle behaviour of shortest paths in the Cayley graph, which needs a different approach and goes beyond of the scope of this article; for a discussion on these problems, see Remark \ref{rem:remarks2}.
\par
The plan of this article is as follows: in Section \ref{sec:notation} we give an introduction to random walks on HNN extensions, summarize some basic properties and present the main results of this article. In Section \ref{sec:genfun} we introduce our main tool, namely generating functions. Section \ref{sec:boundary} describes a boundary (see Proposition \ref{lem:infinite-words}) towards which our random walk converges. In Section \ref{sec:exit-times} we introduce a special Markov chain  (see Proposition \ref{lem:exit-time-chain}) which allows us to track the random walk's path to infinity. This construction finally enables us to derive a formula for the rate of escape w.r.t. the natural word length $\mathfrak{l}$ (see Corollary \ref{cor:speed-word-length}) and existence and formulas for the drift $\lambda_\ell$ for general length functions $\ell$ (see Theorems \ref{thm:alternative-drift-formula} and \ref{thm:drift}). A central limit theorem (see   Theorem \ref{thm:clt}) associated with the word length w.r.t. $\ell$ is derived in Section \ref{sec:clt} and 
analyticity of the drift and the asymptotic variance is then proven in Section \ref{sec:analyticity}, see Theorems \ref{th:analyticity} and \ref{thm:variance-analytic}. Some proofs are outsourced into Appendix A in order to allow a better reading flow.

%We recall the definition of ends of a graph $\mathcal{G}=(V_\mathcal{G},E_\mathcal{G})$ with vertex set $V_\mathcal{G}$, set of edges $E_\mathcal{G}$ and root $o_{\mathcal{G}}$, where we denote by $d_{\mathcal{G}}(\cdot,\cdot)$ the natural graph metric. A infinite geodesic $\pi=(o_\mathcal{G},x_1,x_2,\dots)\in V_\mathcal{G}^{\N_0}$ is an infinite path in $\mathcal{G}$ such that $d_{\mathcal{G}}(o_{\mathcal{G}},x_n)=n$. Two geodesic rays are equivalent if both rays end up in the same connectivity component of $\mathcal{G}\setminus \mathcal{B}_n$, which is the graph which is obtained by removing all vertices $y\in\mathcal{G}$ (and adjacent edges) from $\mathcal{G}$ with $d_{\mathcal{G}}(o_{\mathcal{G}},y)\leq n$ for all $n\in\N$. An equivalence class of rays is called an \textit{end} of $\mathcal{G}$. In group theory, it is well-known that finitely generated groups either have exactly one, two or infinitely many ends.
%\par

\section{HNN Extensions and Random Walks}
\label{sec:notation}

In this section we recall the definition of HNN extensions, summarise some essential properties,  and introduce a natural class of random walks on them. In particular, we introduce length functions on HNN extensions in dependence of some normal form representation of the elements.

\subsection{HNN Extensions of Groups}
Let $G_0=\langle S_0\,|\,  R_0\rangle$ be a finitely generated group with finite set of generators $S_0\subseteq G_0$, relations $R_0$ and identity $e_0$. Let $A,B$ be finite, isomorphic subgroups of $G_0$ and $\varphi: A\to B$ be an isomorphism. Moreover, let $t$ be a symbol (called \textit{stable letter}), which is not an element of $G_0$. Then the \textit{HNN extension} of $G_0$ over $A,B$ w.r.t. $\varphi$ is given by
$$
G := G_0\ast_{\varphi} :=\bigl\langle S_0, t,t^{-1}\,\bigl|\,  R_0, at=t\varphi(a) \textrm{ for } a\in A\bigr\rangle.
$$ 
That is, $G$ consists of all finite words over the alphabet $S_0\cup\{t,t^{-1}\}$, where any two words which can be deduced from each other with the above relations represent the same element of $G_0$. The empty word is denoted by $e$. A natural group operation on $G$ is given by concatenation of words with possible contractions or cancellations in the middle, where $e$ is then the group identity. The definition of $G$ implies that  $G_0\ast_{\varphi}$ is infinite, since  $t^n\in G$ for all $n\in\N$. Note that the relation $at=t\varphi(a)$ implies 
$$
bt^{-1}=t^{-1}\varphi^{-1}(b) \quad \textrm{ for all  } b\in B.
$$
This group structure was introduced by Higman, Neumann and Neumann, whose initials lead to the abbreviation HNN; see \citet{hnn:49}. For further details and explanations of HNN extensions, we refer, e.g., to  \cite{lyndon-schupp}.
\par
In order to help visualize the concept of HNN extensions, we may  think of the Cayley graph of $G$ w.r.t. the generating set $S_0\cup\{t,t^{-1}\}$. This graph is constructed as follows: initially, take the Cayley graph $\mathcal{X}_0$ of $G_0$ with respect to the generating set $S_0$. At each $a\in A$ we attach an additional edge leading to $at=t\varphi(a)$; at those endpoints we attach another copy of $\mathcal{X}_0$, in which we identify $B$ with the already existing vertices $t\varphi(a)$, $a\in A$. This construction is now performed for every coset $g_0A$, $g_0\in G_0$; analogously, we attach new edges from each $b\in B$ to new vertices $bt^{-1}=t^{-1}\varphi^{-1}(b)$, attach then a new copy of $\mathcal{X}_0$ to those endpoints, which are identified with $A$ in the new copy. This construction is then iterated with each coset and each new attached copy of $\mathcal{X}_0$. 
%For $g\in G$, we denote by $|g|$ the length of a shortest path from $e$ to $g$ in the Cayley graph of $G$.

\begin{example}\label{ex:example}
Consider the base group 
$$
G_0=\mathbb{Z}/(2\mathbb{Z})\times \mathbb{Z}/(2\mathbb{Z})=\langle a,b \mid a^2=b^2=e_0,ab=ba\rangle
$$ 
with subgroups $A=\{e_0,a\}, B=\{e_0,b\}$ and isomorphism $\varphi:A\to B$ defined by $\varphi(e_0)=e_0, \varphi(a)=b$. The Cayley graph of the HNN extension is drawn in \mbox{Figure \ref{fig:example}.}
\end{example}

\begin{figure}[h]
\begin{tikzpicture}[scale=0.7]

\coordinate[label=below:$e$] (e) at (0,0);
\coordinate[label=45:$a$] (a) at (2,1);
\coordinate[label=above:$b$] (b) at (0,2);
\coordinate[label=above:$ab$] (ab) at (2,3);

\coordinate[label=below:$t^{-1}$] (s) at (-3,0);
\coordinate[label=above:$t^{-1}a$] (bs) at (-3,2);
\coordinate[label=below:$bt$] (bt) at (3,2);
\coordinate[label=right:$btb$] (btb) at (5,3);

\coordinate[label=above:$bta$] (bta) at (3,4);
\coordinate[label=right:$btba$] (btba) at (5,5);

\coordinate[label=above:$t$] (t) at (3,0);
\coordinate[label=above:$tb$] (tb) at (5,1);
\coordinate[label=below:$ta$] (ta) at (3,-2);
\coordinate[label=160:$tab$] (tab) at (5,-1);

\coordinate[label=right:$t^2b$] (ttb) at (6.5,-2);
\coordinate[label=right:$t^2$] (tt) at (6.5,0);
\coordinate[label=right:$t^2ab$] (ttab) at (5,-3.5);
\coordinate[label=right:$t^2a$] (tta) at (5,-1.5);

\coordinate[label=120:$tbt$] (att) at (8,1);
\coordinate[label=right:$tbtb$] (attb) at (8,-1);
\coordinate[label=above:$tbta$] (atta) at (9.5,2.5);
\coordinate[label=right:$tbtab$] (attab) at (9.5,0.5);

\coordinate[label=right:$tbt^2$] (attt) at (11,1);
\coordinate[label=right:$tbt^2b$] (atttb) at (12.5,2.5);

\coordinate[label=right:$tbt^2a$] (attta) at (11,3);
\coordinate[label=right:$tbt^2ab$] (atttab) at (12.5,4.5);

\coordinate[label=below:$at^{-1}$] (as) at (-1,1);
\coordinate[label=above:$at^{-1}a$] (abs) at (-1,3);
%\coordinate[label=right:$$] () at (,);
%\coordinate[label=right:$$] () at (,);

\draw[red] (e) -- (a);
\draw[orange] (b) -- (ab);
\draw (e) -- (b);
\draw[gray] (a) -- (ab);

\draw[red] (e) -- (a);
\draw (e) -- (b);
\draw[orange]  (b) -- (ab);
\draw[gray]  (a) -- (ab);
\draw[blue]  (e) -- (s);
\draw[blue]  (b) -- (bs);
\draw[red]  (s) -- (bs);
\draw[blue]  (b) -- (bt);
\draw[blue]  (ab) -- (btb);
\draw (bt) -- (btb);
\draw[red] (bt) -- (bta);
\draw[orange]  (btb) -- (btba);
\draw[gray]  (bta) -- (btba);
\draw[blue]  (e) -- (t);
\draw[blue]  (a) -- (tb);
\draw (t) -- (tb);
\draw[red] (t) -- (ta);
\draw[gray]  (ta) -- (tab);
\draw[orange]  (tb) -- (tab);
\draw[blue]  (t) -- (tt);
\draw[blue]  (ta) -- (ttb);

\draw[red] (tt) -- (tta);
\draw (tt) -- (ttb);
\draw[gray]  (tta) -- (ttab);
\draw[orange]  (ttb) -- (ttab);

\draw[blue]  (tb) -- (att);
\draw[blue]  (tab) -- (attb);
\draw[red] (att) -- (atta);
\draw[gray]  (atta) -- (attab);
\draw[orange]  (attb) -- (attab);
\draw (att) -- (attb);

\draw[blue]  (att) -- (attt);
\draw[blue]  (atta) -- (atttb);
\draw[red] (attt) -- (attta);
\draw (attt) -- (atttb);
\draw[gray]  (attta) -- (atttab);
\draw[orange]  (atttb) -- (atttab);
\draw[blue]  (a) -- (as);
\draw[blue]  (ab) -- (abs);
\draw[red] (as) -- (abs);
%\draw (s) -- (-3.5,-0.5);
%
%\draw[gray] (bs) -- (-3.5,1.5);
%\draw () -- ();
%
\fill (e) circle (2pt);
\fill (a) circle (2pt);
\fill (b) circle (2pt);
\fill (ab) circle (2pt);

\fill (s) circle (2pt);
\fill (bs) circle (2pt);
\fill (bt) circle (2pt);
\fill (btb) circle (2pt);
\fill (bta) circle (2pt);
\fill (btba) circle (2pt);

\fill (t) circle (2pt);
\fill (tb) circle (2pt);
\fill (ta) circle (2pt);
\fill (tab) circle (2pt);

\fill (tt) circle (2pt);
\fill (tta) circle (2pt);
\fill (ttb) circle (2pt);
\fill (ttab) circle (2pt);

\fill (att) circle (2pt);
\fill (atta) circle (2pt);
\fill (attb) circle (2pt);
\fill (attab) circle (2pt);
\fill (attt) circle (2pt);

\fill (attta) circle (2pt);
\fill (atttb) circle (2pt);
\fill (atttab) circle (2pt);
\fill (as) circle (2pt);
\fill (abs) circle (2pt);

%\draw[red] (0,0) -- (2,2);
%\draw[orange] (0,2) -- (2,4);
%\draw[blue] (0,0) -- (0,2);
%\draw[blue] (2,2) -- (2,4);
%
%

%
%\fill (0,0) circle (0.25);
%\fill (0,2) circle (0.25);
%\fill (2,2) circle (0.25);
%\fill (2,4) circle (0.25);
%
\end{tikzpicture}
\caption{Part of the Cayley graph of the HNN extension in Example \ref{ex:example}.}
\label{fig:example}
\end{figure}
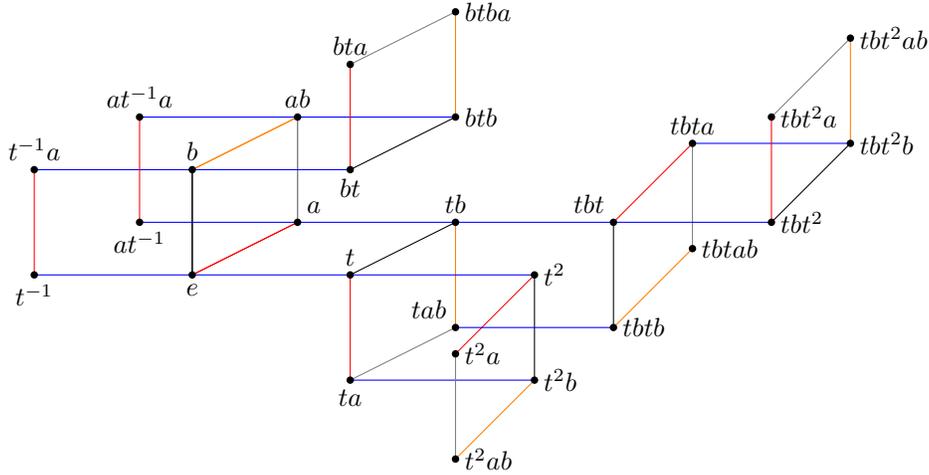

A normal form of the elements of $G_0\ast_{\varphi}$ can be obtained as follows: let $X$ be a set of representatives of the left cosets of $G_0/A$ and $Y$ be a set of representatives of the left cosets of $G_0/\varphi(A)=G_0/B$. We assume w.l.o.g. that $e_0\in X,Y$. Observe that 
$$
t^{-1}e_0t^{}=t^{-1}t\varphi(e_0)=e_0 \quad \textrm{ and } \quad  t^{}e_0t^{-1}=e_0.
$$ 
We get the following normal form expression of each element of $G$:
\begin{lemma}
Each element $g\in G_0\ast_{\varphi}$ has a unique representation of the form
\begin{equation}\label{equ:normalform}
g=g_1 t_1 g_2  t_2 \dots  g_nt_n g_{n+1},
\end{equation}
which satisfies: 
\begin{itemize}
\item $n\in\mathbb{N}_0$, $g_{n+1}\in G_0$, $t_i\in\{t,t^{-1}\}$ for $i\in\{1,\dots,n\}$, 
\item $g_i\in X$, if $t_i=t$, and $g_i\in Y$, if   $t_i=t^{-1}$, 
\item no consecutive subsequences of the form $te_0 t^{-1}$ or $t^{-1}e_0t$.
\end{itemize}
\end{lemma}
\begin{proof}
We prove the claim by induction on the number of letters $t^{\pm 1}$ in \textit{any} given word over the alphabet $S_0\cup\{t^{\pm 1}\}$. First, any $g\in G_0$ is already in the proposed form (\ref{equ:normalform}). Now consider the case of given $g=s_1\dots s_dt^{\varepsilon}s_{d+1}\dots s_{d+e}$ with $d,e\in\N_0$, $\varepsilon\in\{-1,1\}$ and $s_i\in S_0$ for $1\leq i\leq d+e$. If $\varepsilon=1$, we rewrite $s_1\dots s_d=g_1a_1$ with $g_1\in X$ and $a_1\in A$. Then:
$$
g=s_1\dots s_dt s_{d+1}\dots s_{d+e}= g_1a_1ts_{d+1}\dots s_{d+e}=g_1t\underbrace{\varphi(a_1)s_{d+1}\dots s_{d+e}}_{=:g_2\in G_0},
$$
which yields the proposed form. In the case $\varepsilon=-1$, we recall that $bt^{-1}=t^{-1}\varphi^{-1}(b)$ for all $b\in B$. We now write $s_1\dots s_d=g_1b_1$ with $g_1\in Y$ and $b_1\in B$ and obtain the proposed form (\ref{equ:normalform}):
$$
g=s_1\dots s_dt^{-1}s_{d+1}\dots s_{d+e}= g_1b_1t^{-1}s_{d+1}\dots s_{d+e}=g_1t^{-1}\underbrace{\varphi^{-1}(b_1)s_{d+1}\dots s_{d+e}}_{=:g_2\in G_0}.
$$
In particular, the number of letters $t^{\pm 1}$ did \textit{not} increase and the representation is obviously unique.
\par
The induction step follows the same reasoning: consider any word over the alphabet $S_0\cup\{t,t^{-1}\}$ of the form
$$
g=\underbrace{s^{(1)}_1\dots s^{(1)}_{m_1}t_1s^{(2)}_1\dots s^{(2)}_{m_2}t_2\ldots s^{(n-1)}_{m_{n-1}}t_{n-1}}_{=:g'}s^{(n)}_1\dots s^{(n)}_{m_n}t_n\underbrace{s^{(n+1)}_1\dots s^{(n+1)}_{m_{n+1}}}_{=:h},
$$
where $n\geq 2$, $m_1\dots,m_{n+1}\in\N_0$ and $s^{(1)}_1, \dots, s^{(n+1)}_{m_{n+1}}\in S_0$. By induction assumption we can rewrite $g'$ in the form  (\ref{equ:normalform}), say
$$
g'=g'_1t'_1g_2't_2'\dots g'_{k}t'_{k}g'_{k+1} \quad \textrm{ with } k\leq n-1.
$$
We now consider the case $t'_{k}=t$ and $t_n=t$. Rewrite 
$$
g'_{k+1}s^{(n)}_1\dots s^{(n)}_{m_n}=g_na_n
$$
with $g_n\in X$ and $a_n\in A$. Then:
$$
g = g'_1t'_1g_2't_2'\dots g'_{k}t'_{k}g_na_nt_nh = g'_1t'_1g_2't_2'\dots g'_{k}t'_{k}g_nt \varphi(a_n)h
=g'_1t'_1g_2't_2'\dots g'_{k}t'_{k}g_nt h',
$$
where $h':=\varphi(a_n)h\in G_0$, 
that is, we have established the required form (\ref{equ:normalform}). If $t_n=t^{-1}$, rewrite
$$
g'_{k+1}s^{(n)}_1\dots s^{(n)}_{m_n}=g_nb_n
$$
with $g_n\in Y$ and $b_n\in B$. Then:
\begin{eqnarray*}
g &= &g'_1t'_1g_2't_2'\dots g'_{k}t'_{k}g_nb_nt^{-1}h \\
&=& g'_1t'_1g_2't_2'\dots g'_{k}t'_{k}g_nt^{-1} \varphi^{-1}(b_n)h=g'_1t'_1g_2't_2'\dots g'_{k}t'_{k}g_nt^{-1} h',
\end{eqnarray*}
where $h':=\varphi^{-1}(b_n)h\in G_0$.
If $g_n\neq e_0$, we have established the proposed form (\ref{equ:normalform}). In the case $g_n=e_0$, $t'_ke_0t^{-1}$ cancels out, that is,
$g=g'_1t'_1g_2't_2'\dots g'_{k} h'$, which is in the form (\ref{equ:normalform}). The case $t'_{k}=t^{-1}$ follows by symmetry. Uniqueness of the representations follows immediately from the uniqueness of representatives of the cosets. This proves the claim.
\end{proof}
We will refer to the expression in (\ref{equ:normalform}) as \textit{normal form} of the elements of $G$ and we write $\Vert g\Vert$ for the word length of $g\in G$ w.r.t. the normal form. Sometimes we will omit the letter $e_0$ when using normal forms; e.g., instead of writing $e_0te_0t$ we just write $t^2$. In this setting we may omit counting the letter $e_0$ and get the analogous word length. Since this will \textit{not} cause any problems below, we will omit a case distinction whether $e_0$ is counted or not.
Furthermore, we define $[g_1t_1g_2t_2\ldots g_nt_ng_{n+1}]:=g_1t_1g_2t_2\ldots g_nt_n$.
\begin{example}
We revisit Example \ref{ex:example}. In this case we  set $X=\{e_0,b\}$, $Y=\{e_0,a\}$ and obtain, e.g., the following normal forms: 
$$
abt^{-1}=at^{-1}\varphi^{-1}(b)=at^{-1}a,\quad tbt=\varphi^{-1}(b)tt=att=at^2.
$$
Note in Figure \ref{fig:example} the ``rotation'' of the different coloured cosets when pushed along blue $t$-edges.
\end{example}

As a final remark observe that $G$ is amenable if and only if $G_0=A=B$: if $A\subsetneq G_0 $ then the removal of $A\cup B$ from the Cayley graph of $G$ splits the remaining graph into at least three connected components (e.g., $t,t^{-1}, g_0t$ with $g_0\in G_0\setminus A$ are in different components), yielding non-amenability of $G$ (e.g., see \cite[Thm. 10.10]{woess}); if $G_0=A=B$ then the Cayley graph of $G$ has linear growth, yielding amenability of $G$ (e.g., see \cite[Thm 12.2]{woess}).

\subsection{Random Walks on HNN Extensions}
We now introduce a natural class of random walks on HNN extensions arising from random walks on the base group $G_0$.
Let $\mu_0$ be a finitely supported probability measure on $G_0$ whose support generates $G_0$ as a semi-group. W.l.o.g. we assume that $\mathrm{supp}(\mu_0)=S_0$. Furthermore, let be $\alpha,p\in(0,1)$. Then 
\begin{equation}
\mu:=\alpha \cdot \mu_0 + (1-\alpha)\cdot\bigl( p \cdot \delta_t+  (1-p) \cdot \delta_{t^{-1}}\bigr) \label{def:mu}
\end{equation}
is a probability measure on $G$ with $\langle \mathrm{supp}(\mu)\rangle= G$. Let $(\zeta_i)_{i\in\N}$ be an i.i.d. sequence of random variables with distribution $\mu$. A random walk $(X_n)_{n\in\N_0}$ on $G=G_0\ast_{\varphi}$ is then given by
$$
X_0=e,\quad \forall n\geq 1: X_n=\zeta_1\zeta_2\ldots \zeta_n.
$$
For $x,y\in G$, we denote by $p(x,y):=\mu(x^{-1}y)$ the single-step transition probabilities of  $(X_n)_{n\in\N_0}$ and by $p^{(n)}(x,y):=\mu^{(n)}(x^{-1}y)$ the corresponding $n$-step transition probabilities, where $\mu^{(n)}$ is the $n$-fold convolution power of $\mu$. We abbreviate \mbox{$\mathbb{P}_x[\,\cdot\,]:=\mathbb{P}[\,\cdot \, | X_0=x]$.} Analogously, we set $p_0^{(n)}(x_0,y_0):=\mu_0^{(n)}(x_0^{-1}y_0)$ for $x_0,y_0\in G_0$ and $n\in\mathbb{N}$.
We have the following characterisation for the recurrence/transience behaviour of random walks on HNN extensions:
\pagebreak[4]
\begin{lemma}\label{lem:recurrence}
The random walk on $G$ is recurrent if and only if $A=B=G_0$ and $p=\frac12$.
\end{lemma}
\begin{proof}
Assume that $A=B=G_0$ and $p=\frac12$. Then every normal form has the form $t^ng_0$ or $t^{-n}g_0$ with $n\in\N_0$ and $g_0\in G_0$. Define $\psi: G\to\mathbb{Z}$ by $\psi(t^ng_0):=n$, $\psi(t^{-n}g_0):=-n$ respectively. Then $(X_n)_{n\in\N_0}$ is recurrent if and only if $\bigl(\psi(X_n)\bigr)_{n\in\N_0}$ is recurrent. But $\bigl(\psi(X_n)\bigr)_{n\in\N_0}$ is just a delayed simple random walk on $\mathbb{Z}$, which is obviously recurrent.
\par
If we assume $p\neq\frac12$ but  $A=B=G_0$, then we get transience of $(X_n)_{n\in\N_0}$.
\par
Assume now that $A\subsetneq G_0$, that is $|X|,|Y|\geq 2$. Then $G$ is non-amenable, which yields together with \cite[Cor.12.5]{woess}  that the spectral radius given by $\limsup_{n\to\infty} p^{(n)}(e,e)^{1/n}$ is strictly smaller than $1$, that is, the random walk on $G$ is transient.
\end{proof}
Consider the Cayley graph of $G$ w.r.t. the generating set $S_0\cup\{t,t^{-1}\}$, which induces a natural metric $d(\cdot,\cdot)$.
We have:
\begin{lemma}
For nearest neighbour random walks on $G$, the rate of escape w.r.t. the natural graph metric
$$
\mathfrak{s}=\lim_{n\to\infty} \frac{d(e,X_n)}{n}
$$
exists. Moreover, we have $\mathfrak{s}>0$ if and only if $(X_n)_{n\in\N_0}$ is transient.
\end{lemma}
\begin{proof}
Existence is well-known due to Kingman's subadditive ergodic theorem, see \cite{kingman}. 
Obviously, $\mathfrak{s}>0$ implies transience. Vice versa, by Lemma \ref{lem:recurrence}, transience is equivalent to $A\subset G_0$ or $G_0=A=B$ with $p\neq\frac12$.  
If $A\subsetneq G_0$ then $G$ is non-amenable  and we obtain a spectral radius strictly smaller than $1$, see \cite[Cor. 12.5]{woess}. This yields $\mathfrak{s}>0$; see \cite[Thm. 8.14]{woess}. If $G_0=A=B$ and $p\neq\frac12$ then we can project the random walk onto $\Z$ (see proof of Lemma \ref{lem:recurrence}), which gives $\mathfrak{s}=(1-\alpha)|2p-1|>0$.
\end{proof}
If $G_0$ is finite and $\mathrm{supp}(\mu_0)=G_0$, then one can regard $(X_n)_{n\in\N_0}$ as a random walk on a regular language over a finite alphabet, for which existence and analyticity of $\lim_{n\to\infty} \Vert X_n\Vert/n$ follows from the formulas in \cite{gilch:08}. If $G$ is hyperbolic then analyticity of $\mathfrak{s}$ and the asymptotic entropy follows from the work of \cite{gouezel:15}. Note that, in general, HNN extensions need not to be hyperbolic.
\par
Since we are interested in transient random walks, we exclude from now on the case that both $A=B= G_0$ and $p= \frac12$ hold. 
 
\subsection{Generalised Length Functions on $G$} 
Let $\ell:G_0\cup\{t,t^{-1}\}\to [0,\infty)$ be a  function, which plays the role of a \textit{generalised length} or \textit{weight function} for each letter. For $g=g_1 t_1 g_1  t_2 \dots  g_nt_n g_{n+1}$ in normal form as in (\ref{equ:normalform}), we extend $\ell$ to a ``length function'' on $G$ via
$$
\ell(g_1 t_1 g_1  t_2  \dots  g_nt_n g_{n+1}) :=\sum_{k=1}^n \Bigl( \ell(g_k)+\ell(t_k)\Bigr) + \ell(g_{n+1}).
$$
Note that the natural word length is obtained by setting $\ell(\cdot)=1$.
If there is a non-negative constant number $\lambda_\ell$ such that
$$
\lambda_\ell = \lim_{n\to\infty} \frac{\ell(X_n)}{n}\quad \textrm{almost surely},
$$
then $\lambda_\ell$ is called the \textit{rate of escape} (or \textit{drift} or \textit{asymptotic word length}) \textit{w.r.t. the length function $\ell$}. One aim of this paper is to show existence of this limit in the transient case under the following growth assumption on $\ell$, which will be needed as an integrability condition later. We say that $\ell$ is of \textit{polynomial growth} if there are some $\kappa\in\N$ and $C>0$ such that  $\ell(g_0)\leq C\cdot |g_0|^\kappa$ for all $g_0\in G_0$, where 
\begin{eqnarray*}
|g_0|&=&\min\{m\in\N \mid \exists s_1,\dots,s_m\in G_0: g=s_1\dots s_m\}\\
&=&\min\{m\in\N_0 \mid p^{(m)}(e,g)>0  \}.
\end{eqnarray*}

\begin{remark}\label{rem:remarks}\normalfont
While existence of the rate of escape w.r.t. the natural graph metric given by the almost sure constant limit $\mathfrak{s}=\lim_{n\to\infty} d(e,X_n)/n$ is well-known due to Kingman's subadditive ergodic theorem, existence of $\lambda_\ell$ is not given a-priori for arbitrary length functions $\ell$: e.g., if $g_1,g_2,g_3\in G_0$ with  $g_3=g_1^{-1}g_2$, \mbox{$\ell(g_1)=\ell(g_3)=1$} and $\ell(g_2)=3$, then
$$
\ell(g_1)+\ell(g_1^{-1}g_2) <\ell(g_2);
$$
that is, subadditivity does not necessarily hold, and therefore Kingman's subadditive ergodic theorem is not applicable.
\end{remark}

As an application of generalized length functions we can construct an upper bound for the asymptotic entropy, see Corollary \ref{cor:entropy}. We note that, in general, the natural graph metric can \textit{not} be expressed via length functions.
We refer to Remark \ref{rem:remarks2} for further discussion on the obstacles when studying rate of escape w.r.t. the natural graph metric $\mathfrak{s}$.

\subsection{Main Results}

We summarize the main results of this article. The first main result shows existence of the rate of escape w.r.t. length functions $\ell$ of polynomial growth. As we will see in Section \ref{sec:boundary}, the prefixes of $X_n$ of increasing length stabilize (that is, the prefixes of increasing length are not changed any more after some finite time). We denote by $\e{1}$, $\e{2}$ respectively, the random time from which on the first two letters of $X_n$, the first four letters of $X_n$ respectively, stabilize. The involved expectations (denoted by $\mathbb{E}_\pi[\cdot]$) in the following theorem are taken w.r.t. some invariant probability distribution $\pi$, see (\ref{def:pi}) in Section \ref{sec:exit-times} for more details.
\begin{theorem}\label{thm:alternative-drift-formula}
Let $(X_n)_{n\in\N_0}$ be a transient random walk on $G$ governed by $\mu$ as defined in (\ref{def:mu}), and let $\ell$ be a length function of polynomial growth.
Then there exists a positive constant $\lambda_\ell$ such that
\begin{equation}\label{equ:alternative-drift-formula}
\lambda_\ell=\lim_{n\to\infty}\frac{\ell(X_n)}{n}=\frac{\mathbb{E}_\pi\bigl[\ell([X_{\e{2}}])-\ell([X_{\e{1}}])\bigr]}{\mathbb{E}_\pi[\e{2}-\e{1}]}>0 \quad \textrm{almost surely.}
\end{equation}
\end{theorem}
In particular, the formula holds also for the rate of escape w.r.t. the natural word length (that is, if $\ell(g_0)=\ell(t^{\pm 1})=1$ for all $g_0\in G_0$). We remark that   \cite{haissinsky-mathieu-mueller12} derived a similar formula for random walks on hyperbolic surface groups. 
The proof of this theorem is given in Section \ref{sec:exit-times}, where the main steps are as follows: we construct a positive-recurrent Markov chain which is derived from the random times when new pairs of letters in the prefix of $X_n$ stabilize, see Proposition \ref{lem:exit-time-chain}. 
 This Markov chain traces the random walk's path to infinity. The crucial point here is that these random times are no stopping times which destroys the Markov property of $(X_n)_{n\in\N_0}$ when conditioning on these random times. Having shown some necessary integrability property in Lemma \ref{lem:sum-finite}, we are able to derive  a formula for the rate of escape w.r.t. the natural word length (see Proposition \ref{thm:t-drift} and Corollary \ref{cor:speed-word-length}), from which we can finally deduce existence of $\lambda_\ell$ in Theorem \ref{thm:drift} and the formula in Theorem \ref{thm:alternative-drift-formula}. Let me remark that the theorem generalizes the result of \cite{gilch:08} for infinite $G_0$.
\par
The next main result is a central limit theorem for the word length w.r.t. $\ell$. For this purpose, we use the Markov chain introduced in Proposition \ref{lem:exit-time-chain} for the definition of regeneration times $(T_n)_{n\in\N_0}$ (defined in (\ref{def:regeneration-times})), which are special random times (no stopping times!) at which the random walk $(X_n)_{n\in\N_0}$ stabilizes further letters in its prefix in a specific way. 

\begin{theorem}\label{thm:clt}
Let $(X_n)_{n\in\N_0}$ be a transient random walk on $G$ governed by $\mu$ as defined in (\ref{def:mu}), and let $\ell$ be a length function of  polynomial growth.
Then the rate of escape w.r.t. $\ell$ satisfies 
$$
\frac{\ell(X_n)- n\cdot \lambda_\ell }{\sqrt{n}}\xrightarrow{\mathcal{D}} N(0,\sigma^2),
$$
where %$\displaystyle \lambda=\frac{\mathbb{E}\bigl[\ell([X_{T_1}])-\ell([X_{T_0}])\bigr]}{\mathbb{E}[T_1-T_0]}$ and  
$\displaystyle \sigma^2=\frac{\mathbb{E}\bigl[\bigl(\ell(X_{T_1})-\ell(X_{T_0})-(T_1-T_0)\lambda_\ell\bigr)^2\bigr]}{\mathbb{E}[T_1-T_0]}$.
\end{theorem}
The proof is given in Section \ref{sec:clt}. The idea of the proof is to cut the trajectory of $(X_n)_{n\in\N_0}$ into i.i.d. subsequences. 
Lemmas \ref{lem:tau-expmom} and \ref{lem:T-expmom} show that the time increments between two consecutive regeneration times have exponential moments. From this follows then the proposed central limit theorem.
\par
The third main result demonstrates that $\lambda_\ell$ varies real-analytically in terms of probability measures of constant support. Let $S_0=\{s_1,\dots,s_{d}\}$ generate $G_0$ as a semigroup and denote by 
$$
\mathcal{P}_0(S_0)=\Bigl\lbrace (p_1,\dots,p_d) \,\Bigl|\, \forall i\in\{1,\dots,d\}: p_i>0, \sum_{j=1}^d p_j=1\Bigr\rbrace
$$ 
the set of all probability measures $\mu_0$ on $S_0$, where  $\mu_0(s_i)=p_i$ for $i\in\{1,\dots,d\}$.
Hence, we may regard $\lambda_\ell$ as a mapping  $(\mu_0,\alpha,p)\mapsto \lambda_\ell(\mu_0,\alpha,p)$.
\begin{theorem}\label{th:analyticity}
Let $S_0\subseteq G_0$ be finite and generating $G_0$ as a semi-group, and consider transient random walks on $G$ governed by probability measures of the form \mbox{$\mu=\alpha \mu_0+(1-\alpha)\bigl(p\delta_t+(1-p)\delta_{t^{-1}}\bigr)$} with $\mathrm{supp}(\mu_0)=S_0$.
%Let $\mu_0$ be a finitely generated probability measure on $G_0$ whose support $S_0=\mathrm{supp}(\mu_0)$ generates $G_0$ as a semigroup. 
Furthermore, let $\ell$ be a length function of at most polynomial growth.  Then the mapping
$$
\lambda_\ell: \mathcal{P}_0(S_0)\times (0,1) \times (0,1) \to\mathbb{R}: \mu=(\mu_0,\alpha,p)\mapsto \lambda_\ell(\mu_0,\alpha,p)
$$
is real-analytic.
\end{theorem}
For the proof of the theorem in Section \ref{sec:analyticity} we will use the formula for $\lambda_\ell$ given in \ref{equ:lambda-T-formula}. We show in Lemmas \ref{lem:ET0} and \ref{lem:EellT0} that both nominator and denominator can be rewritten as multivariate power series in terms of $\mu_0,\alpha,p$ with sufficiently large radii of convergence. In the same way we obtain our last main result:

\begin{theorem}\label{thm:variance-analytic}
The asymptotic variance $\sigma^2$ from Theorem \ref{thm:clt} varies real-analyti\-cally when considered as a multivariate power series, that is, the mapping 
$$
(\mu_0,\alpha,p)\mapsto \sigma^2=\sigma^2(\mu_0,\alpha,p)
$$
varies real-analytically.
\end{theorem}

Concerning the rate of escape $\mathfrak{s}$ w.r.t. the natural graph metric, we obtain a special case if $A=B$ is normal in $G_0$:
\begin{corollary}
Assume that $A=B\trianglelefteq G_0$ and $\varphi=\mathrm{id}_A$. Then Theorems \ref{thm:clt}, \ref{th:analyticity} and \ref{thm:variance-analytic} hold also, if $\ell(g)=d(e,g)$, $g\in G$, is the distance of $g$ to $e$ w.r.t. the natural graph metric in the Cayley graph of $G$.
\end{corollary}
\begin{proof}
It is easy to show that $G/A$ is isomorphic to the free product $(G_0/A)\ast \mathbb{Z}$. In this case one can project the random walk $(X_n)_{n\in\N_0}$ onto $G/A$, for which a formula for the rate of escape w.r.t. the natural graph metric is given in \cite{gilch:07}. If $\ell: G_0\to \N$ describes the distance of the elements of $G_0$ to $e_0$ in $G_0$ w.r.t. the natural graph metric and if we set $\ell(t^{\pm 1}):=1$, then the extension of $\ell(\cdot)$ to $G$ describes the distance of any $g\in G$ to $e$ w.r.t. the natural graph metric in the associated Cayley graph of $G$.
\end{proof}

%\pagebreak[4]

\section{Generating Functions}
\label{sec:genfun}

In this section we introduce several important probability generating functions, which are power series with some probabilities of interest as coefficients. These generating functions will  play a technical key role in our proofs.
\par
For $x,y\in G$ and $z\in\mathbb{C}$, the \textit{Green function} is defined as
$$
G(x,y|z):=\sum_{n\geq 0} p^{(n)}(x,y)\,z^n.
$$
For any $M\subseteq G_0$, we write $tM:=\{tm\mid m\in M\}$ and $t^{-1}M:=\{t^{-1}m\mid m\in M\}$. 
For $a\in A$, $b\in B$, define the \textit{generating functions w.r.t. the first visit of $G_0$} when starting at $tb$,  or at $t^{-1}a$ respectively,
\begin{eqnarray*}
\eta(tb|z) &:=& \sum_{n\geq 1} \mathbb{P}_{tb}\bigl[X_n\in G_0,X_{n-1}\in tB,\forall m\in\{1,\dots,n-2\}:X_m\notin G_0\bigr]z^n,\\
\eta(t^{-1}a|z)&:=& \sum_{n\geq 1} \mathbb{P}_{t^{-1}a}\left[\begin{array}{c} X_n\in G_0, X_{n-1}\in t^{-1}A,\\ \forall m\in\{1,\dots,n-2\}:X_m\notin G_0\end{array}\right]z^n.
\end{eqnarray*}
Furthermore, we define
\begin{eqnarray*}
\xi(tb|z)&:= &1-\eta(tb|z),\\
\xi(t^{-1}a|z)&:= &1-\eta(t^{-1}a|z).
\end{eqnarray*}
In particular, we have
\begin{eqnarray*}
\xi(tb)&:= & \xi(tb|1)= \mathbb{P}_{tb}[\forall n\in \N: X_n\notin G_0]= \mathbb{P}_{tb}[\forall n\in \N: X_n\notin A],\\
\xi(t^{-1}a)&:= & \xi(t^{-1}a|1)= \mathbb{P}_{t^{-1}a}[\forall n\in \N: X_n\notin G_0]= \mathbb{P}_{t^{-1}a}[\forall n\in \N: X_n\notin B].
\end{eqnarray*}

Observe that all paths from $tb$ to $G_0$ have to pass through $A$: in order to walk from any $tg$, where $g\in G_0$, to $g_0\in G_0$ one has to eliminate the $t$-letter, which is only possible if $g\in B$; in this case
$$
tgt^{-1} = tt^{-1}\varphi^{-1}(g)=\varphi^{-1}(g)\in A.
$$
Analogously, each path from $t^{-1}a$ to $G_0$ has to pass through $B$.

%
% OLD DEFINITION
%$$
%\xi(tb):= \mathbb{P}_{tb}[\forall n\in N: X_n\notin G_0] = 1-\sum_{b_1\in B} \sum_{n\geq 0} \mathds{1}_{b}^T (\mathbb{G}_B(z)+\mathbb{K}_1(z))^n \mathds{1}_{b_1}\cdot (1-\alpha)\cdot (1-p)
%$$
%and
%$$
%\xi(t^{-1}h):= \mathbb{P}_{th}[\forall n\in N: X_n\notin G_0] = 1-\sum_{\tilde h\in H} \sum_{n\geq 0} \mathds{1}_h^T (\mathbb{G}_H(z)+\mathbb{H}_{-1}(z))^n \mathds{1}_{\tilde h} \cdot (1-\alpha)\cdot p
%$$

\begin{lemma}\label{lem:xi>0}
Assume $A,B\subsetneq G_0$. Then we have for 
 all $a\in A$ and  $b\in B$:
$$
\xi(tb)>0 \ \textrm{ and } \ \xi(t^{-1}a)>0.
$$
\end{lemma}
\begin{proof}
Since the random walk $(X_n)_{n\in\N_0}$ on $G$ is assumed to be transient and $A$ and $B$ are finite, we have
\begin{equation}\label{equ:A-visits}
\mathbb{P}[A \textrm{ is visited infinitely often}]= \mathbb{P}[B \textrm{ is visited infinitely often}]=0.
\end{equation}
Assume now for a moment that $\xi(tb)=0$ and $\xi(t^{-1}a)=0$ for all $a\in A, b\in B$. This implies that $\eta(tb|1)=\eta(t^{-1}a|1)=1$ for all $a\in A$ and $b\in B$. Hence, for all $x\in X$, $y\in Y$, we have
$$
\mathbb{P}_{xtb}[\exists n\in\N: X_n\in G_0] = \eta(tb)=1=
\eta(t^{-1}a)=\mathbb{P}_{yt^{-1}a}[\exists n\in\N: X_n\in G_0];
$$
that is, every time when the random leaves $G_0$ to some point $xtb$ or $yt^{-1}a$, it returns almost surely to $G_0$. This gives together with vertex-transitivity of the random walk:
$$
\mathbb{P}[G_0 \textrm{ is visited infinitely often}]=\mathbb{P}_t[tG_0 \textrm{ is visited infinitely often}]=1.
$$
This in turn yields that 
\begin{eqnarray*}
\mathbb{P}[tG_0 \textrm{ is visited infinitely often}] &\geq&
\mathbb{P}[X_1=t, tG_0 \textrm{ is visited infinitely often}]\\
&=&  (1-\alpha)\cdot p \cdot \mathbb{P}_t[tG_0 \textrm{ is visited inf. often}]>0.
\end{eqnarray*}
Therefore, the event that both $G$ and $tG_0$ are visited infinitely often has positive probability. Since every path from $tG_0$ to $G_0$ has to pass through $A$, the event that  $A$ is visited infinitely often has also positive probability, which now gives a contradiction to the transience behaviour in (\ref{equ:A-visits}).
\par
Assume now that $\xi(tb_0)>0$ for some $b_0\in B$ and let be $b\in B$. Then, due to irreducibility of $\mu_0$ there is some $n_0\in\N$ with $p_0^{(n_0)}(b,b_0)=\mu_0^{(n_0)}(b^{-1}b_0)>0$. This yields:
\begin{eqnarray*}
\xi(tb)&\geq & \mathbb{P}_{tb}[X_1,\dots,X_{n_0-1}\in tG_0,X_{n_0}=tb_0,\forall n\geq 1: X_n\notin G_0]\\
&\geq&  \alpha^{n_0}p_0^{(n_0)}(b,b_0)\xi(tb_0)>0.
\end{eqnarray*}
Choose now any $x\in X\setminus\{e_0\}$ (observe that $A\subsetneq G_0$ implies $|X|\geq 2$) and let be $a\in A$.
Then there is some $n_1\in\N$ with $p_0^{(n_1)}\bigl(a,x\varphi^{-1}(b_0)\bigr)>0$. We bound $\xi(t^{-1}a)$ by paths which start at $t^{-1}a$, go directly to $t^{-1}x$, then to $t^{-1}xt$ without any further modification of the first three letters afterwards:\begin{eqnarray*}
\xi(t^{-1}a)& \geq& \mathbb{P}_{t^{-1}a}\left[\substack{X_1,\dots,X_{n_1-1}\in t^{-1}G_0,X_{n_1}=t^{-1}x\varphi^{-1}(b_0),\\ X_{n_1+1}=t^{-1}x\varphi^{-1}(b_0)t,\forall n\geq n_1+1: X_n\notin t^{-1}G_0}\right]\\
&\geq & \alpha^{n_1}p_0^{(n_1)}\bigl(a,x\varphi^{-1}(b_0)\bigr) \cdot (1-\alpha)\cdot p\cdot \xi(tb_0)>0.
\end{eqnarray*}
Here, recall that $\varphi^{-1}(b_0)t=tb_0$. This finishes the proof.
%Analogously, if we start assuming $\xi(t^{-1}a_0)>0$ for some $a_0\in A$ then the reasoning follows analogously.
%
% Choose some $g_0=[g_0]b_0 \in G_0$ with $g_0\notin B$, $b_0\in B$. Due to irreducibility of $\mu_0$ there is some $n_0\in\N$ with $p_0^{(n_0)}(tb,tg_0)>0$.
%Then:
%$$
%\xi(tb) \leq 1-\alpha^{n_0}\cdot p_0^{(n_0)}(tb,tg_0) \cdot 
%<1
%$$
%Due to transience of our random walk, we have 
%$$
%P_{tb}\bigl[ \exists n\in\N: X_n\in A\bigr]
%$$
%
\end{proof}

%\begin{lemma}
%For all $a\in A$, $b\in B$,
%$$
%\xi(tb)>0 \ \textrm{ and } \ \xi(t^{-1}a)>0.
%$$
%\end{lemma}
%\begin{proof}
%Since the random walk $(X_n)_{n\in\N_0}$ on $G$ is transient and $A$ and $B$ are finite, we have
%$$
%\sum_{a\in A} G(g,a)<\infty \quad \forall g\in G_0\setminus A, \ \sum_{b\in B} G(g,b)<\infty \quad \forall g\in G_0\setminus B \ \textrm{respectively}.
%$$
%Thus, $\mathbb{P}_{g}[\exists n\in\mathbb{N}: X_n\in H]<1$. Choose $m\in\N$ such that
%$$
%p_0^{(m)}(h,g)>0.
%$$
%Then:
%$$
%\xi(tb)=\mathbb{P}_{tb}[\forall n\in\mathbb{N}: X_n\notin G_0] \geq
%\alpha^m \cdot p^{(m)}(tb,tg)\cdot \mathbb{P}_g[\forall n\in\mathbb{N}: X_n\notin B] >0.
%$$
%\end{proof}

An analogous statement is obtained in the remaining case for transient random walks.
\begin{lemma}\label{lem:xi>0-2}
Consider the case $A=B=G_0$ and $p\neq \frac12$. Let be $a,b\in G_0$. Then $\xi(tb)>0$ and $\xi(t^{-1}a)=0$, if $p>\frac12$,  and $\xi(tb)=0$ and $\xi(t^{-1}a)>0$, if $p<\frac12$.
\end{lemma}
\begin{proof}
In the case $p>\frac12$ the stochastic process $\bigl(\psi(X_n)\bigr)_{n\in\N_0}$ from the proof of Lemma \ref{lem:recurrence} tends to $+\infty$ almost surely, yielding $\xi(tb)>0$ and $\xi(t^{-1}a)=0$ for all $a,b\in G_0=A=B$. The case $p<\frac12$ follows by symmetry.
\end{proof}

The following property will be essential in the proofs of the upcoming sections.

\begin{lemma}\label{lem:radius-of-convergence}
The common radius of convergence $R$ of $G(g_1,g_1|z)$, $g_1,g_2\in G$, is strictly bigger than $1$. Moreover, the generating functions $\eta(\cdot|z)$ and $\xi(\cdot|z)$ have also radii of convergence of at least $R$. 
%and the spectral radius $\varrho=\limsup_{n\to\infty} p^{(n)}(e,e)^{1/n}$ of $(X_n)_{n\in\N_0}$ is strictly smaller than $1$.
\end{lemma}
\begin{proof}
First, we remark that  all Green functions must have the same radius of convergence $R$ due to irreducibility of the underlying random walk. Since we consider only transient random walks, Lemma \ref{lem:recurrence} implies that either \mbox{$A,B\subsetneq G_0$} or $p\neq \frac12$ must hold.  
 \par
If $A,B\subsetneq G_0$ then $G$ is non-amenable, implying  that the spectral radius satisfies  \mbox{$\varrho=\limsup_{n\to\infty}p^{(n)}(e,e)^{1/n} <1$;} see, e.g., \cite[Cor. 12.5]{woess}. This in turn implies $R=\varrho^{-1}>1$.
\par
 The proof of the fact that $G(e,e|z)$ has also in the case $p\neq \frac12$ a  radius of convergence   strictly bigger than $1$ is outsourced to Lemma \ref{lem:p-neq-1/2} in the Appendix.
\par
It remains to consider $\eta(\cdot|z)$ and $\xi(\cdot|z)$. For $b\in B$, choose $n_b\in\N$ with $\mu_0^{(n_b)}\bigl(\varphi^{-1}(b)\bigr)>0$, which is possible due to irreducibilty of $\mu_0$. Then for real $z>0$:
\begin{eqnarray*}
\sum_{a\in A} G(e,a|z)&\geq&
\sum_{n\geq n_b+2}\mathbb{P}\left[ 
\substack{X_{n_b}=\varphi^{-1}(b),\forall m\in\{1,\dots,n_b-1\} X_{m}\in G_0,\\
X_{n_b+1}=\varphi^{-1}(b)t,  X_n\in G_0}
\right]\cdot z^n\\
&=&  \alpha^{n_b}\cdot \mu_0^{(n_b)}\bigl(\varphi^{-1}(b)\bigr)\cdot (1-\alpha)\cdot p\cdot z^{n_b+1}\cdot \eta(tb|z),
\end{eqnarray*}
where the right hand sides describes all paths, where one walks in $n_b$ steps inside $G_0$ to $\varphi^{-1}(b)$, then walks to $\varphi^{-1}(b)t=tb$ and returns afterwards to the set $A$. The above inequality implies that $\eta(tb|z)$ has also radius of convergence of at least $R$ for all $b\in B$; analogously for $\eta(t^{-1}a|z)$. The same holds for $\xi(tb|z)$ and $\xi(t^{-1}a|z)$ by definition.
\end{proof}

In the proofs later the following lemma will be a convenient tool:
\begin{lemma}\label{lem:K-convergence}
The generating function
$$
\mathcal{K}(z) :=  \sum_{g_0\in G_0} G(e,g_0|z)=\sum_{g_0\in G_0}\sum_{n\geq 0} p^{(n)}(e,g_0)z^n
$$
has radius of convergence strictly bigger than $1$. In particular, $\mathcal{K}(z)$ is arbitrarily often differentiable at $z=1$.
\end{lemma}
\begin{proof}
For $n\in\N$, define
$$
\zeta_n := \mathbb{P}\bigl[X_n\in G_0,\forall m\in\{1,\dots,n-1\}: X_m\notin G_0 \bigr],
$$
the probability of starting in $e$ and returning to $G_0$ at time $n$ without making any steps within $G_0$ until time $n$. Recall that this implies $X_n\in A\cup B$. Set
$$
\mathcal{G}_0(z):=\sum_{n\geq 0}\zeta_n\cdot z^n, \quad z\in\mathbb{C}.
$$
We decompose every path from $e=e_0$ to any $g_0\in G_0$ by the number $m$ of steps performed w.r.t. $\mu_0$: set $\mathbf{s}(0):=0$ and define 
$$
\mathbf{s}(k):=\inf\{n>\mathbf{s}(k-1)\mid X_{n-1},X_n\in G_0\}\textrm{ for  } k\geq 1.
$$ 
In other words, at times $\mathbf{s}_k$ the random walk makes a step within $G_0$. For all $n\in\N$, we can write
\begin{eqnarray*}
&&\sum_{g_0\in G_0}p^{(n)}(e,g) = \sum_{g_0\in G_0} \sum_{m=0}^n \mathbb{P}\bigl[ \mathbf{s}(m)\leq n,\mathbf{s}(m+1)>n,X_n=g_0\bigr]\\
&=& \zeta_n +
\sum_{g_0\in G_0} \sum_{m=1}^n \sum_{\substack{t_1,\dots,t_m\in\N: \\ t_1<t_2<\ldots <t_m\leq n} }\mathbb{P}\left[\begin{array}{c} \mathbf{s}(1)=t_1,\dots,\mathbf{s}(m)=t_m,\\ \mathbf{s}(m+1)>n,X_n=g_0\end{array}\right]\\
&=&\zeta_n+  \sum_{m=1}^n \sum_{\substack{t_1,\dots,t_m\in\N: \\ t_1<t_2<\ldots <t_m\leq n} }\bigl(\zeta_{t_1-1}\cdot \alpha\bigr)\cdot  \bigl(\zeta_{(t_2-t_1)-1}\cdot \alpha\bigr)\cdot \ldots \cdot  \bigl(\zeta_{(t_m-t_{m-1})-1}\cdot \alpha\bigr).
\end{eqnarray*}
%By decomposing every path from $e=e_0$ to any $g_0\in G_0$ by the number $m$ of steps performed w.r.t. $\mu_0$, 
This allows us to rewrite $\mathcal{K}(z)$ for $z\in\mathbb{C}$ in the interior of the domain of convergence:
 $$
\mathcal{K}(z) := \sum_{g_0\in G_0}\sum_{n\geq 0} p^{(n)}(e,g_0)z^n = \mathcal{G}_0(z) \cdot  \sum_{m\geq 0} \bigl( \mathcal{G}_0(z)\cdot \alpha\cdot z\bigr)^m.
$$
Observe that, for real $z>0$, we have
$$
\mathcal{G}_0(z)= \sum_{n\geq 0} \zeta_n\, z^n \leq \sum_{n\geq 0} \P[X_n\in A\cup B]\, z^n
=
\sum_{h\in A\cup B} G(e,h|z). 
$$
Since $A\cup B$ is finite and the generating functions $G(e,h|z), h\in A\cup B$, have common radius of convergence strictly bigger than $1$ due to Lemma \ref{lem:radius-of-convergence}, $\mathcal{G}_0(z)$ has also radius of convergence strictly bigger than $1$.
\par
Consider now
$$
q(z):=\mathcal{G}_0(z) \cdot \alpha \cdot z.
$$
Observe that starting at $e_0$ (or equivalently due to transitivity, starting at any $g_0 \in G_0 $) the probability of returning to $G_0$ followed directly by a step performed w.r.t to $\mu_0$ is given by
$q(1)$, that is,
$$
\P[\mathbf{s}(1)<\infty]=\mathbb{P}\bigl[\exists m\in\N_0: X_m,X_{m+1}\in G_0\bigr]=q(1).
$$
%Due to Lemmas \ref{lem:xi>0} and \ref{lem:xi>0-2} there is at least one element $t_1h_1\in tB\cup t^{-1}A$ with $\eta(t_1h_1)<1$. Therefore,
%\begin{eqnarray*}
%q(1) & \leq & \mathbb{P}[X_1\in G_0] + \mathbb{P}[X_1\in tB\cup t^{-1}A, \exists n\in\N: X_n\in G_0
%
%\end{eqnarray*}
Since $\mathcal{G}_0(1)=\mathbb{P}[\exists n\in\N: X_n\in G_0]$ we have  $q(1)=\mathcal{G}_0(1)\cdot \alpha \leq \alpha <1$.
%Since $\xi(tb)<1$ and $\xi(t^{-1}a)<1$ for all $a\in A, b\in B$, we have $q(1)<1$. 
Moreover, $q(z)$ has radius of convergence $R(q)>1$. Since $q(z)$ as a power series is continuous, we can choose $\rho \in \bigl(1,R(q)\bigr)$ with $q(\rho)<1$. Then:
$$
\mathcal{K}(\rho) =   \mathcal{G}_0(\rho) \cdot  \sum_{m\geq 0} q(\rho)^m =  \frac{\mathcal{G}_0(\rho)}{1-q(\rho)}<\infty.
$$
Hence, $\mathcal{K}(z)$ has radius of convergence of at least $\rho>1$.
\end{proof}

\section{Boundary of the Random Walk}
\label{sec:boundary}

In this section we describe a natural boundary of the random walk on $G$. Define
$$
\mathcal{B}:=\left\lbrace g_1t_1 g_2t_2\dots \,\biggl|\, \substack{g_1,g_2,\dots \in X\cup Y, \ t_1,t_2,\dots \in \{t,t^{-1}\}, \\   t_i=t\Rightarrow g_i\in X, t_i=t^{-1}\Rightarrow g_i\in Y, g_i=e_0\Rightarrow t_{i-1}t_i\neq e }\right\rbrace \subset \bigl(X\cup Y\cup\{t,t^{-1}\}\bigr)^{\N},
$$
the set of infinite words in normal form.  \cite{woess-isr} showed that an irreducible random walk with finite range on an HNN extension with $A\subsetneq G_0$ converges to a random infinite word in $\mathcal{B}$. Nonetheless, we give a precise mathematical related statement and a general proof (which covers also the case $A=G_0$) of this convergence towards $\mathcal{B}$, because the proofs are short and help the reader to get a better understanding of the structure of HNN extensions.
\par
The \textit{$t$-length} of a word $g=g_1 t_1 g_2  t_2  \dots  g_n t_n g_{n+1}$ in normal form in the sense of (\ref{equ:normalform}) is defined as
\begin{equation}\label{def:t-length}
|g|_t:=n.
\end{equation}
We make the first observation that each copy of $G_0$ is visited finitely often only:
\begin{lemma}\label{lem:copy-finite-visit}
Let be $g_1t_1\dots g_kt_k\in G_0$ in normal form. Then the set  $g_1t_1\dots g_kt_kG_0$ is  visited finitely often almost surely. 
\end{lemma}
\begin{proof}
%Since the random walk on $G$ is assumed to be transient, the finite set $g_1t_1^{\varepsilon_1}\dots g_kt_k^{\varepsilon_k}H$ is visited finitely often only.
First, we consider the case $A,B\subsetneq G_0$. Let be $n_1,n_2,\ldots\in\N$  the instants of time at which the random walk $(X_n)_{n\in\N_0}$ visits the set $g_1t_1\dots g_kt_kG_0$. Suppose that the random walk is at $g=g_1t_1\dots g_kt_kg_{k+1}^{(j)}$, $g_{k+1}^{(j)}\in G_0$, at some time $n_j$. Then the probability of walking from $g$ to $gt_k$ with no further revisit of $g_1t_1\dots g_kt_kG_0$ is at least 
$$
(1-\alpha)\cdot \min\{p,1-p\}\cdot \min_{h\in H}\xi(t_k h)>0,
$$
where $H=A$ if $t_k=t^{-1}$, and $H=B$ if $t_k=t$; here, we used Lemma \ref{lem:xi>0}.
Therefore, a geometric distribution argument shows that there are almost surely only finitely many indices $m\in\N$ with $X_{m}\in g_1t_1\dots g_kt_kG_0$. This proves the claim in the case $A,B\subsetneq G_0$.
\par
In the case $A=B=G_0$ and $p\neq \frac12$ the claim follows directly from transience of the projections $\bigl(\psi(X_n)\bigr)_{n\in\N_0}$ in the proof of Lemma \ref{lem:recurrence} and finiteness of $A$ \mbox{and $B$.}
\end{proof}
The last lemma motivates the definition of the \textit{exit times} $\e{k}$, $k\in\N$, as
$$
\e{k}:=\min \bigl\lbrace m\in\N_0 \mid \forall n\geq m: |X_n|_t\geq k\bigr\rbrace.
$$
Let be $g_\infty=g_1t_1 g_2t_2\ldots\in\mathcal{B}$ and denote by $X_n\wedge g_\infty$ the common prefix of $X_n$ and $g_\infty$, that is, if $X_n=g'_1t'_1 g'_2t'_2\dots g'_k t_k'g'_{k+1}$, then 
$$
X_n\wedge g_\infty = g_1t_1 \dots g_lt_l,
$$
where $l=\max\{i\in\N \mid  g_1t_1 \dots g_it_i=g'_1t'_1 \dots g'_it_i\}$.
We say that a realisation $(x_0,x_1,\dots)\in G^{\mathbb{N}_0}$ of $(X_n)_{n\in\N_0}$ converges to  $g_\infty$ if $\lim_{n\to\infty}|x_n\wedge g_\infty|=\infty$. 
Now we are able to show that $\mathcal{B}$ is a natural  boundary of the random walk towards which $(X_n)_{n\in\N_0}$ converges. 
\begin{proposition}\label{lem:infinite-words}
For all $k\in\N$, $\e{k}<\infty$ almost surely. Moreover, the random walk $(X_n)_{n\in\N_0}$ converges almost surely to some $\mathcal{B}$-valued random variable $X_\infty$.
\end{proposition}
\begin{proof}
It is sufficient to prove that, for all $m\in\N$, there is some index $N_m$ such that we have $|X_{n}|_t\geq m$ for all $n\geq N_m$. We prove this claim by induction.
By Lemma \ref{lem:copy-finite-visit}, the set $G_0$ is almost surely visited finitely often, that is, there is some minimal, almost surely finite random time $\e{1}$ such that $|X_{n}|_t\geq 1$ for all $n\geq \e{1}$. In particular, the first two letters of $X_n$ are stabilized and will \textit{not} change for $n\geq \e{1}$.
\par
Assume now that there is some finite random time $\e{m}$ such that  \mbox{$|X_{n}|_t\geq m$} for all $n\geq \e{m}$. This implies that the prefix of $X_n$ of $t$-length $m$ is constant, that is, there is some word $g=g_1t_1\dots g_mt_m$ such that $X_n$ starts with $g$ for all $n\geq \e{m}$. Once again by Lemma \ref{lem:copy-finite-visit}, the set $gG_0$ is almost surely visited finitely often only, that is, there is some almost surely finite random time $\e{m+1}\in\N$ such that $|X_{n}|_t\geq m+1$ for all $n\geq \e{m+1}$. But this means that there are $g_{m+1}\in X\cup Y$, $t_{m+1}\in\{t,t^{-1}\}$ such that $X_n$ starts with $gg_{m+1}t_{m+1}$ for all $n\geq \e{m+1}$.
This finishes the proof.
\end{proof}
In \cite{woess-isr} it is shown that $(\mathcal{B},\nu)$ is a model for the Poisson boundary, where $\nu$ is the hitting probability of $\mathcal{B}$, that is, for measureable $B\subset\mathcal{B}$,  $\nu(B)$ is the probability that $(X_n)_{n\in\N_0}$ converges to some element in $B$.

\section{Existence of the Rate of Escape w.r.t. $\ell$}
\label{sec:exit-times}

In this section we derive existence of the rate of escape w.r.t. the length function $\ell$ by introducing a new Markov chain which tracks the random walk's way towards the boundary $\mathcal{B}$; compare with \cite{gilch:07,gilch:08,gilch:11}.
\par
Recall the definition of the \textit{exit times} $\e{k}$, $k\in\N$, from the last section. 
%$$
%\e{k}:=\min \bigl\lbrace m\in\N_0 \mid \forall n\geq m: |X_n|_t\geq k\bigr\rbrace.
%$$
By Proposition \ref{lem:infinite-words}, $\e{k}<\infty$ almost surely for all $k\in\N$. The \textit{increments} are defined as
$$
\i{k}:=\e{k}-\e{k-1}.
$$
Furthermore,  if $X_{\e{k}}=g_1t_1\dots g_kt_kh$, where $h\in B$, if $t_k=t$, and $h\in A$, if $t_k=t^{-1}$, then we set
$$
\W_{k}:=g_kt_kh.
$$
Set
\begin{eqnarray*}
\mathcal{D}&:=&\{ gth \mid g\in X, h\in B\} \cup \{ gt^{-1}h \mid g\in Y, h\in A\},\\
\mathbb{D} &:=& \left\lbrace (gt'h,n)\in\mathcal{D}\times\N \,\biggl|\,
\exists g_1t_1h_1\in \mathcal{D}:
\mathbb{P}\left[\begin{array}{c}X_{\e{1}}=g_1t_1h_1,\\ X_{\e{2}}=g_1t_1gt'h,\mathbf{i}_2=n\end{array}\right]>0\right\rbrace.
\end{eqnarray*}
Since the events $[X_{\e{k}}=n]$, $n\in\N$, depend on the future after  $\e{k}$, the exit times are no stopping times. Hence, conditioning the random walk $(X_n)_{n\in\N_0}$ on exit times destroys the Markov property. However, we make the following crucial observation:
\begin{proposition}\label{lem:exit-time-chain}
$(\W_{k},\mathbf{i}_k)_{k\in\N}$ is an irreducible, aperiodic Markov chain on $\mathbb{D}$ with transition probabilities
\begin{eqnarray*}
&&\P[\W_{k+1}=w_2t_2h_2,\mathbf{i}_{k+1}=n\mid \W_{k}=w_1t_1h_1,\mathbf{i}_{k}=m]\\
&=&\begin{cases} 
\frac{\xi(t_2h_2)}{\xi(t_1h_1)}\cdot \P_{t_1h_1}\Bigl[\substack{X_n=t_1w_2t_2h_2,|X_{n-1}|_t=1,\\ \forall n'<n: |X_{n'}|_t\geq 1 }\Bigr], & \mathrm{if }\ t_1w_2t_2\neq e,\\
0, &\mathrm{otherwise},
\end{cases}
\end{eqnarray*}
where $(w_1t_1h_1,m),(w_2t_2h_2,n)\in\mathbb{D}$.
%
%w_1\in X,t_1=t,h_1\in B$ or $w_1\in Y,t_1=t^{-1},h_1\in A$ and $w_2\in X,t_2=t,h_2\in B$ or $w_2\in Y,t_2=t^{-1},h_2\in A$ with $t_1w_2t_2\neq e$.
\end{proposition}
\begin{proof}
Let be $(w_1t_1h_1,n_1),\dots,(w_{k+1}t_{k+1}h_{k+1},n_{k+1})\in\mathbb{D}$ such that this sequence satisfies
$\mathbb{P}[\forall j\in\{1,\dots,k\}: \mathbf{W}_j=w_jt_jh_j,\mathbf{i}_j=n_j]>0$. In particular, the words $w_1t_1\dots w_jt_jh_{j}$, $j\leq k+1$, are in normal form in the sense of (\ref{equ:normalform})
Then:
\begin{eqnarray*}
&& \P\bigl[\W_{1}=w_1t_1h_1,\mathbf{i}_{1}=n_1, \dots,  \W_{k}=w_{k} t_{k}h_{k},\mathbf{i}_{k}=n_{k}\bigr] \\[1ex]
&=&\P\left[\substack{X_{\e{1}}=w_1t_1h_1,\mathbf{i}_{1}=n_1,X_{\e{2}}=w_1t_1w_2t_2h_2,\mathbf{i}_{2}=n_2,\\  \dots,  X_{\e{k}}=w_1t_1\dots w_{k} t_{k}h_{k},\mathbf{i}_{k}=n_{k}}\right] \\[1ex]
&=& \P\left[ \substack{\forall j\in\{1,\dots,k\} \forall m\in\{0,\dots,n_j\}:\\[1ex]
|X_{n_1+\dots+n_{j-1}+m}|_t\geq j-1, |X_{n_1+\dots+n_{j}-1}|_t= j-1, X_{n_1+\dots + n_j}=w_1t_1\dots w_jt_jh_j} \right]\\
&&\quad
\cdot \P_{w_1t_1\dots w_kt_kh_k}\bigl[\forall n\geq 1: |X_{n_1+\dots + n_k+n}|_t\geq k\bigr]\\[1ex]
&=& \P\left[ \substack{\forall j\in\{1,\dots,k\} \forall m\in\{0,\dots,n_j\}:\\
|X_{n_1+\dots+n_{j-1}+m}|_t\geq j-1, |X_{n_1+\dots+n_{j}-1}|_t= j-1, X_{n_1+\dots + n_j}=w_1t_1\dots w_jt_jh_j} \right]
 \cdot \xi(t_kh_k).
\end{eqnarray*}
The last equation is due to transitivity (group invariance) of our random walk $(X_n)_{n\in\N_0}$. Analogously,
\begin{eqnarray*}
&& \P[\W_{1}=w_1t_1h_1,\mathbf{i}_{1}=n_1, \dots,  \W_{k+1}=w_{k+1} t_{k+1}h_{k+1},\mathbf{i}_{k+1}=n_{k+1}] \\[1ex]
&=&\P\left[\substack{X_{\e{1}}=w_1t_1h_1,\mathbf{i}_{1}=n_1,X_{\e{2}}=w_1t_1w_2t_2h_2,\mathbf{i}_{2}=n_2, \dots, \\ X_{\e{k+1}}=w_1t_1\dots w_{k+1} t_{k+1}h_{k+1},\mathbf{i}_{k+1}=n_{k+1}}\right] \\[1ex]
&=& \P\left[ \substack{\forall j\in\{1,\dots,k\} \forall m\in\{0,\dots,n_j\}:\\
|X_{n_1+\dots+n_{j-1}+m}|_t\geq j-1, |X_{n_1+\dots+n_{j}-1}|_t= j-1, X_{n_1+\dots + n_j}=w_1t_1\dots w_jt_jh_j} \right]\\
&&\quad
\cdot \P_{w_1t_1\dots w_kt_kh_k}\left[\substack{\forall m\in\{0,\dots,n_{k+1}\}: |X_{n_1+\dots+n_{k}+m}|_t\geq k, |X_{n_1+\dots+n_{k+1}-1}|_t= k, \\ X_{n_1+\dots + n_{k+1}}=w_1t_1\dots w_{k+1}t_{k+1}h_{k+1} }\right]\\
&&\quad 
\cdot \P_{w_1t_1\dots w_{k+1}t_{k+1}h_{k+1}}\bigl[\forall n\geq 1: |X_{n_1+\dots + n_{k+1}+n}|_t\geq k+1\bigr]\\[1ex]
&=& \P\left[ \substack{\forall j\in\{1,\dots,k\} \forall m\in\{0,\dots,n_j\}:\\
|X_{n_1+\dots+n_{j-1}+m}|_t\geq j-1, |X_{n_1+\dots+n_{j}-1}|_t= j-1, X_{n_1+\dots + n_j}=w_1t_1\dots w_jt_jh_j} \right]\\
&&\quad
\cdot \P_{w_1t_1\dots w_kt_kh_k}\left[\substack{\forall m\in\{0,\dots,n_{k+1}\}: |X_{n_1+\dots+n_{k}+m}|_t\geq k, \\ |X_{n_1+\dots+n_{k+1}-1}|_t= k, \\ X_{n_1+\dots + n_{k+1}}=w_1t_1\dots w_{k+1}t_{k+1}h_{k+1}}\right]
\cdot  \xi(t_{k+1}h_{k+1}).
\end{eqnarray*}
Hence, transitivity of the random walk yields yields once again:
\begin{eqnarray*}
&& \P\left[\W_{k+1}=w_{k+1} t_{k+1}h_{k+1},\mathbf{i}_{k+1}=n_{k+1}\,\biggl| \,\substack{\W_{1}=w_1t_1h_1,\mathbf{i}_{1}=n_1, \dots,  \\ \W_{k}=w_{k} t_{k}h_{k},\mathbf{i}_{k}=n_{k}}\right] \\
&=& \frac{\xi(t_{k+1}h_{k+1})}{\xi(t_kh_k)} \P_{t_kh_k}\left[\begin{array}{c} X_{n_{k+1}}=t_kw_{k+1}t_{k+1}h_{k+1},|X_{n_{k+1}-1}|_t=1,\\ \forall n<n_{k+1}: |X_n|_t\geq 1\end{array}\right].
\end{eqnarray*}
From the  formula above follows that $\mathrm{supp}(\mathbf{W}_k,\mathbf{i}_k)=\mathbb{D}$ for $k\geq 2$: indeed, for any $(g_1t_1h_1,n_1)\in\mathbb{D}$, there exists some $g_0t_0h_0\in\mathcal{D}$ with $g_0 \neq e_0$ such that 
$$
\P\bigl[X_{\e{1}}=g_0t_0h_0,X_{\e{2}}=g_0t_0g_1t_1h_1,\mathbf{i}_2=n_1\bigr]>0, 
$$
yielding
$$
\mathbb{P}\left[\begin{array}{c} X_{\e{1}}=t_0,X_{\e{2}}=t_0^2,\dots,X_{\e{k-2}}=t_0^{k-2},\\ X_{\e{k-1}}=t_0^{k-2}g_0t_0h_0,X_{\e{k}}=t_0^{k-2}g_0t_0g_1t_1h_1,\mathbf{i}_2=n_1\end{array}\right]>0,
$$
that is, $(g_1t_1h_,n_1)\in\mathrm{supp}(\mathbf{W}_k,\mathbf{i}_k)$.
\par
For irreducibility and aperiodicity, it suffices to show that any $(g_1t_1h_1,n_1)\in\mathbb{D}$ can be reached from any other $(g_0t_0h_0,n_0)\in\mathbb{D}$ in two steps. First, we consider the case $t_1=t$. Let be $g_1t_1h_1=xtb$ with $x\in X$ and $b\in B$ and choose \mbox{$\bar g_0\bar t_0\bar h_0,\bar g_1\bar t_1\bar h_1\in\mathcal{D}$} with $\bar g_1\neq e_0$ such that 
\begin{eqnarray*}
\P[X_{\e{1}}=\bar g_0\bar t_0\bar h_0,X_{\e{2}}=\bar g_0 \bar t_0 g_0t_0h_0,\mathbf{i}_2=n_0]& >& 0\  \textrm{ and  }\\
 \P[X_{\e{1}}=\bar g_1\bar t_1\bar h_1,X_{\e{2}}=\bar g_1\bar t_1 xtb,\mathbf{i}_2=n_1]&>&0;
\end{eqnarray*}
compare with definition of $\mathbb{D}$ and recall transitivity of the random walk. Take any $m\in\N$ such that \mbox{$\mu_0^{(m)}\bigl(h_0^{-1}\bar g_1 \varphi^{\delta}(\bar h_1)\bigr)>0$,} where $\delta:=1$, if $\bar t_1=t^{-1}$,  and $\delta:=-1$, if $\bar t_1=t$; then for all $k\geq 2$:
$$
\mathbb{P}\left[\begin{array}{c} 
X_{\e{1}}=\bar t_0, X_{\e{2}}=\bar t_0^2,\dots ,  X_{\e{k-2}}=\bar t_0^{k-2},\\
X_{\e{k-1}}=\bar t_0^{k-2}\bar g_0\bar t_0 \bar h_0, 
X_{\e{k}}=\bar t_0^{k-2}\bar g_0\bar t_0g_0t_0h_0,
\mathbf{i}_k=n_0,\\
X_{\e{k+1}}=\bar t_0^{k-2}\bar g_0\bar t_0g_0t_0\bar g_1\bar t_1\bar h_1,\mathbf{i}_{k+1}=m+1,\\
X_{\e{k+2}}=\bar t_0^{k-2}\bar g_0\bar t_0g_0t_0\bar g_1\bar t_1xtb,\mathbf{i}_{k+2}=n_1\end{array}\right]>0.
$$
%If $t_0=t^{-1}$ or $\bar t=t^{-1}$ (this case can just occur if $A,B\subsetneq G_0$, because $t_1=t_0=\bar t$ must hold in the case $A=B=G_0$), then there are some $x'\in X\setminus\{e_0\}$, $m\in\N$ such that
%$$
%\mathbb{P}\left[\begin{array}{c}X_{\e{2}}=g_0t^{-1}h_0,\mathbf{i}_1=n_0,X_{\e{3}}=g_0t^{-1}x'\bar t\bar h, \\ \mathbf{i}_2=m,X_{\e{4}}=g_0t^{-1}x'txtb,\mathbf{i}_3=n_1\end{array}\right]>0.
%$$
Hence, we have proven that each element of $\mathbb{D}$ can be reached in two steps from any other state if $t_1=t$.
The case $t_1=t^{-1}$ is shown analogously. This finishes the proof.
\end{proof}
Observe that, for all $(w_1t_1h_1,m),(w_2t_2h_2,n)\in\mathbb{D}$, the transition probabilities of $(\W_{k},\mathbf{i}_k)_{k\in\N}$ in Lemma \ref{lem:exit-time-chain}
\begin{eqnarray*}
&&q\bigl((w_1t_1h_1,m),(w_2t_2h_2,n) \bigr)\\
&:=& 
\begin{cases}
\P\bigl[\W_{k+1}=w_2t_2h_2, \mathbf{i}_{k+1}=n\mid \W_k=w_1t_1h_1, \mathbf{i}_k=m\bigr], & \textrm{if } t_1w_2t_2\neq e\\
0,& \textrm{otherwise}
\end{cases}
\end{eqnarray*}
depend only on $t_1h_1$, $w_2t_2h_2$ and $n$, but \textit{not} on $w_1$ and $m$. If $\W_k=w_kt_kh_k$ then set 
$$
\mathbf{h}_k:=t_kh_k
$$
and define
$$
\mathcal{D}_0:=\{ th \mid  h\in B\} \cup \{ t^{-1}h \mid h\in A\}.
$$ 
Note that $\mathbf{h}_k$ can take only finitely many different values. It is easy to see that $(\mathbf{h}_k)_{k\in\N}$ forms an irreducible Markov chain on $\mathcal{D}_0$ with transition probabilities
$$
q_\mathbf{h}(t_1h_1,t_2h_2)= \begin{cases}
\displaystyle \sum_{x\in X,n\in\N} q\bigl( (e_0th_1,m), (xth_2,n)\bigr), & \textrm{if } t_1=t_2=t, \\
\displaystyle\sum_{y\in Y\setminus\{e_0\},n\in\N} q\bigl( (e_0th_1,m), (yt^{-1}h_2,n)\bigr), & \textrm{if } t_1=t_2^{-1}=t, \\
\displaystyle\sum_{y\in Y,n\in\N} q\bigl( (e_0t^{-1}h_1,m), (yt^{-1}h_2,n)\bigr), & \textrm{if } t_1=t_2=t^{-1}, \\
\displaystyle\sum_{x\in X\setminus\{e_0\},n\in\N} q\bigl( (e_0t^{-1}h_1,m), (xth_2,n)\bigr), & \textrm{if } t_1=t_2^{-1}=t^{-1},\\
%\sum_{y\in X,n\in\N} q\bigl( (xt_1h_1,m), (yt_2h_2,n)\bigr), & \textrm{if } t_1=t \textrm{ and } t_1yt_2\neq e\\
%\sum_{y\in Y,n\in \N} q\bigl( (xt_1h_1,m), (yt_2h_2,n)\bigr), & \textrm{if } t_1=t^{-1}\textrm{ and } t_1yt_2\neq e\\
%0, & \textrm{if } \textrm{ and } t_1yt_2=e. 
\end{cases}
$$
where the quantities on the left do not depend on   $m$ as long as $(e_0t_1h_1,m)\in\mathbb{D}$. Due to the finite state space of $(\mathbf{h}_k)_{k\in\N}$, this process is positive recurrent and possesses an invariant probability measure $\nu_{\mathbf{h}}$. For $(w_1t_1h_1,n)\in\mathbb{D}$, set
\begin{equation}\label{def:pi}
\pi(w_1t_1h_1,n):=
\sum_{t_0h_0\in \mathcal{D}_0}\nu_{\mathbf{h}}(t_0h_0)q\bigl( (e_0t_0h_0,m), (w_1t_1h_1,n)\bigr).
%\begin{cases}
%\sum_{h_1\in B} \nu_{\mathbf{h}}(h_1)  q\bigl( (xt_1h_1,m), (yt_0h_0,n)\bigr), & \textrm{if } t_1=t\\
%\sum_{h_1\in A} \nu_{\mathbf{h}}(h_1)  q\bigl( (xt_1h_1,m), (yt_0h_0,n)\bigr), & \textrm{if } t_1=t^{-1}\\
%\end{cases}
\end{equation}
\begin{lemma}
$\pi$ is an invariant probability measure of $(\W_{k},\mathbf{i}_k)_{k\in\N}$. In particular, $(\W_{k},\mathbf{i}_k)_{k\in\N}$ is a positive recurrent Markov chain on $\mathbb{D}$.
\end{lemma}
\begin{proof}
Let be $(w_1t_1h_1,n)\in \mathbb{D}$. Then:
\begin{eqnarray*}
&&\sum_{(w_0t_0h_0,m)\in \mathbb{D}} \pi(w_0t_0h_0,m) q\bigl( (w_0t_0h_0,m), (w_1t_1h_1,n)\bigr)\\
&=& \sum_{(w_0t_0h_0,m)\in \mathbb{D}} \sum_{t'h'\in \mathcal{D}_0} \nu_{\mathbf{h}}(t'h') q\bigl( (t'h',m'), (w_0t_0h_0,m)\bigr) q\bigl( (w_0t_0h_0,m), (w_1t_1h_1,n)\bigr) 
\end{eqnarray*}
\begin{eqnarray*}
&=& \sum_{h_0\in B} \underbrace{\sum_{t'h'\in \mathcal{D}_0} \nu_{\mathbf{h}}(t'h') \underbrace{\sum_{\substack{x\in X,\\ m\in\N}}q\bigl( (e_0t'h',m'), (xth_0,m)\bigr)}_{=q_\mathbf{h}(t'h',th_0)}}_{=\nu_\mathbf{h}(th_0)} \underbrace{q\bigl( (xth_0,m), (w_1t_1h_1,n)\bigr)}_{=q\bigl( (e_0th_0,m_0), (w_1t_1h_1,n)\bigr)} \\
&& +\sum_{h_0\in A} \underbrace{\sum_{t'h'\in \mathcal{D}_0} \nu_{\mathbf{h}}(t'h') \underbrace{\sum_{\substack{y\in Y,\\ m\in\N}}q\bigl( (e_0t'h',m'), (yt^{-1}h_0,m)\bigr)}_{=q_\mathbf{h}(t'h',t^{-1}h_0)}}_{=\nu_\mathbf{h}(t^{-1}h_0)} \underbrace{q\bigl( (yt^{-1}h_0,m), (w_1t_1h_1,n)\bigr)}_{=q\bigl( (e_0t^{-1}h_0,m_0), (w_1t_1h_1,n)\bigr)} \\
&=& \sum_{t_0h_0\in \mathcal{D}_0}\nu_\mathbf{h}(t_0h_0) q\bigl( (e_0t_0h_0,m), (w_1t_1h_1,n)\bigr) = \pi(w_1t_1h_1,n).
\end{eqnarray*}
Above we have chosen $m\in\N$ such that $(e_0t^{\pm 1}h_0,m)\in\mathbb{D}$; the exact value of $m$, however, does not play any role.
\end{proof}
Now we can prove:
\begin{lemma}\label{lem:sum-finite}
For all $s\in\N$,
$$
\Lambda_s:=\sum_{(w_1t_1h_1,m)\in \mathbb{D}} m^s\cdot \pi(w_1t_1h_1,m)<\infty.
$$
\end{lemma}
\begin{proof}
We prove finiteness only in the case $s=1$.
Set $H(t):=A$ and $H(t^{-1}):=B$. Rewriting the above sum yields:
\begin{eqnarray*}
&& \sum_{(w_1t_1h_1,m)\in \mathbb{D}} m\cdot \pi(w_1t_1h_1,m)\\
&=& \sum_{(w_1t_1h_1,m)\in \mathbb{D}} \sum_{t_0h_0\in \mathcal{D}_0} \nu_{\mathbf{h}}(t_0h_0) \cdot q\bigl( (e_0t_0h_0,m_0),(w_1t_1h_1,m)\bigr)\cdot m \\
&=&  \sum_{t_0h_0\in \mathcal{D}_0} \nu_{\mathbf{h}}(t_0h_0) \sum_{(w_1t_1h_1,m)\in \mathbb{D}} q\bigl( (e_0t_0h_0,m_0),(w_1t_1h_1,m)\bigr)\cdot m \\
&=& \sum_{t_0h_0\in \mathcal{D}_0} \nu_{\mathbf{h}}(t_0h_0) \sum_{\substack{(w_1t_1h_1,m)\in \mathbb{D}:\\ t_0w_1t_1\neq e}}m\cdot  \frac{\xi(t_1h_1)}{\xi(t_0h_0)} \P_{t_0h_0}\left[\substack{\forall m\leq n: X_m\notin H(t_0), X_{m-1}\in t_0G_0,\\ X_m=t_0w_1t_1h_1}\right]\\
&\leq & \sum_{t_0h_0\in \mathcal{D}_0} \nu_{\mathbf{h}}(t_0h_0) \sum_{m\in\N}m\cdot  \frac{\max_{t_1h_1\in \mathcal{D}_0} \xi(t_1h_1)}{\xi(t_0h_0)} \underbrace{\P_{t_0h_0}[X_{m-1}\in t_0G_0]}_{=\P[X_{m-1}\in G_0]}\\ %\dot (1-\alpha) \max\{p,1-p\}\\
&\leq & \sum_{t_0h_0\in \mathcal{D}_0} \nu_{\mathbf{h}}(t_0h_0) \frac{\max_{t_1h_1\in \mathcal{D}_0} \xi(t_1h_1)}{\xi(t_0h_0)} \sum_{m\geq 1} m\cdot \P[X_{m-1}\in G_0]\\
&\leq & \sum_{t_0h_0\in \mathcal{D}_0} \nu_{\mathbf{h}}(t_0h_0) \frac{\max_{t_1h_1\in \mathcal{D}_0} \xi(t_1h_1)}{\xi(t_0h_0)}\cdot \frac{\partial}{\partial z}\bigl[ z\cdot \mathcal{K}(z)\bigr]\Bigl|_{z=1}  < \infty,
%&=& \sum_{t'h'\in D_0} \nu_{\mathbf{h}}(th')  \sum_{\varepsilon\in\{\pm\},h\in H(\varepsilon)} \sum_{m\geq 1} m\cdot  \frac{\xi(t_0h_0)}{\xi(th')} \P_{h'}[\forall m\leq n: X_m\notin Ht'^{-1}, X_{m-1}\in G_0]\\
%&&\quad \dot p(g,gt^\varepsilon h) \mathds{1}_{[g\neq o]\cup [t_0t^\varepsilon\neq o]}(g,\varepsilon)\\
%&=&  \sum_{t'h'\in D_0} \nu_{\mathbf{h}}(t'h')  \sum_{m\geq 1} \sum_{\varepsilon\in\{\pm 1\},h\in H(\varepsilon)} 
%\frac{\partial}{\partial z}\bigl[ z\cdot \mathcal{K}_\varepsilon(h'|z)\bigr]\Bigl|_{z=1}  < \infty
\end{eqnarray*}
due to Lemma \ref{lem:K-convergence}. In the case $s>1$, the reasoning is analogously, where we use the fact that $\mathcal{K}(z)$ is arbitrarily often differentiable at $z=1$.
\end{proof}
We set $\Lambda:=\Lambda_1$. The last lemma leads to our first results, where we follow a reasoning, which was similarly used also, e.g., in \cite{woess2} and \cite{gilch:07,gilch:08,gilch:11}.
\begin{proposition}\label{thm:t-drift}
The rate of escape w.r.t. the $t$-length exists and satisfies
$$
\lim_{n\to\infty} \frac{|X_n|_t}{n}=\frac{1}{\Lambda}\quad \textrm{almost surely}.
$$
\end{proposition}
\begin{proof}
First, observe that the ergodic theorem for positiv recurrent Markov chains together with Lemma \ref{lem:sum-finite} yields
$$
\frac{\e{k}}{k}=\frac{1}{k} \sum_{l=1}^k \mathbf{i}_l \xrightarrow{k\to\infty} \Lambda \quad \textrm{ almost surely}.
$$
Define $\mathbf{k}(n):=\max\{k\in\N \mid \e{k}\leq n\}$. Then we obtain  almost surely:
$$
1\leq \frac{n}{\e{\mathbf{k}(n)}} \leq \frac{\e{\mathbf{k}(n)+1}}{\e{\mathbf{k}(n)}} = \frac{\e{\mathbf{k}(n)+1}}{\mathbf{k}(n)+1}\frac{\mathbf{k}(n)+1}{\e{\mathbf{k}(n)}}\xrightarrow{n\to\infty} 1,
$$
hence
$$
\lim_{n\to\infty} \frac{\e{\mathbf{k}(n)}}{n}=1 \textrm{ almost surely.}
$$
This yields:
$$
0\leq \frac{|X_n|_t-|X_{\e{\mathbf{k}(n)}}|_t}{n} \leq  \frac{n-\e{\mathbf{k}(n)}}{n}= 1-\frac{\e{\mathbf{k}(n)}}{n}\xrightarrow{n\to\infty} 0 \ \textrm{ almost surely.}
$$
Finally, we obtain:
\begin{equation}
\frac{|X_n|_t}{n}=\underbrace{\frac{|X_n|_t-|X_{\e{\mathbf{k}(n)}}|_t}{n}}_{\to 0}+ \underbrace{\frac{|X_{\e{\mathbf{k}(n)}}|_t}{\mathbf{k}(n)}}_{=1} \underbrace{\frac{\mathbf{k}(n)}{\e{\mathbf{k}(n)}}}_{\to \Lambda^{-1}}\underbrace{\frac{\e{\mathbf{k}(n)}}{n}}_{\to 1}
\xrightarrow{n\to\infty} \frac{1}{\Lambda}\quad \textrm{almost surely.}\label{equ:thm-equ}
\end{equation}
\end{proof}
\begin{corollary}\label{cor:speed-word-length}
The rate of escape w.r.t. the normal form word length exists and satisfies
$$
\lim_{n\to\infty} \frac{\Vert X_n\Vert}{n}=\frac{2}{\Lambda}.
$$
\end{corollary}
\begin{proof}
This is an immediate consequence of Proposition \ref{thm:t-drift} together with the fact that 
$$
2|g|_t-1 \leq \Vert g\Vert \leq 2|g|_t+1\quad \textrm{ for all } g\in G. 
$$
We remark that existence follows also from Kingman's subadditive ergodic theorem.
\end{proof}

Now we extend Proposition \ref{thm:t-drift} to existence of the rate of escape w.r.t. arbitrary length functions $\ell$ of polynomial growth. 
%Here, we call $\ell$ of \textit{polynomial growth}, if there are some $\kappa\in\N$, $C\in\mathbb{R}$ such that $\ell(g_0)\leq C\cdot |g_0|^\kappa$ for all $g_0\in G_0$, where 
%$$
%|g_0|:=\min\{n\in\N\mid  s_1,\dots,s_n\in S_0 \textrm{ with } g=s_1\dots s_n\}
%$$ 
%is the minimal generator word length of $g_0$. 
For $(w_0t_0h_0,m)\in\mathbb{D}$, define $\widetilde{\ell}(w_0t_0h_0,m):=\ell(w_0t_0)$ and set
$$
\Delta:=\int \widetilde{\ell}\, d\pi = \sum_{(w_0t_0h_0,m)\in\mathbb{D}} \ell(w_0t_0)\cdot  \pi(w_0t_0h_0,n) <\infty,
$$
where finiteness follows from Lemma \ref{lem:sum-finite}.
We obtain:
\begin{theorem}\label{thm:drift}
Let $\ell\not\equiv 0$ be a length function on $G_0\cup\{t,t^{-1}\}$ which is of polynomial growth. % $\pi$-integrable, that is, $\int l(gt_0)\, \pi\bigl(d(gt_0h_0,n) \bigr)<\infty$.
Then the rate of escape w.r.t. $\ell$ exists and is given by the almost sure positive constant number
$$
\lambda_\ell=\lim_{n\to\infty} \frac{\ell(X_n)}{n}=\frac{\Delta}{\Lambda}>0.
$$
%where 
%$$
%\Delta:=\int l(gt_0)\, \pi\bigl(d(gt_0h_0,n) \bigr).
%$$
\end{theorem}
\begin{proof}
We can write $X_{\e{\mathbf{k}(n)}}=g_1t_1\dots g_{\mathbf{k}(n)}t_{\mathbf{k}(n)}g'_{\mathbf{k}(n)+1}$ in normal form as in (\ref{equ:normalform}). Observe that $g'_{\mathbf{k}(n)+1}\in A\cup B$.
Then the ergodic theorem for positive recurrent Markov chain yields
$$
\lim_{n\to\infty}\frac{\ell(X_{\e{\mathbf{k}(n)}})}{\mathbf{k}(n)}
= \lim_{n\to\infty}\frac{1}{\mathbf{k}(n)}\sum_{i=1}^{\mathbf{k}(n)}\ell(g_it_i)
\xrightarrow{n\to\infty} \Delta\ \textrm{ almost surely.}
$$
By assumption on $\ell$, there are $C>0$ and $\kappa\in\N$ such that $\ell(g_0)\leq C\cdot |g_0|^\kappa$ for all $g\in G_0$.
By Lemma \ref{lem:sum-finite}, we have $\Lambda_\kappa=\lim_{n\to\infty} \frac1k \sum_{j=1}^k \mathbf{i}_j^\kappa<\infty$ almost surely. Setting $M:=\max\{\ell(t),\ell(t^{-1})\}$ we get almost surely:
\begin{eqnarray*}
0&\leq & \frac{\ell(X_n)-\ell(X_{\e{\mathbf{k}(n)}})}{n}\leq \frac{C\cdot (n-\e{\mathbf{k}(n)})^\kappa +M\cdot (n-\e{\mathbf{k}(n)})}{n}\\
&\leq & \frac{C\cdot (\e{\mathbf{k}(n)+1}-\e{\mathbf{k}(n)})^\kappa+M\cdot (\e{\mathbf{k}(n)+1}-\e{\mathbf{k}(n)})}{n}\\
&=&  \frac{C\cdot \mathbf{i}_{\mathbf{k}(n)+1}^\kappa+M\cdot \mathbf{i}_{\mathbf{k}(n)+1}}{n}
\xrightarrow{n\to\infty} 0.
\end{eqnarray*}
The rest follows as in (\ref{equ:thm-equ}). Observe that $\Delta>0$ if $\ell\not\equiv 0$.
\end{proof}

We are now able to prove Theorem \ref{thm:alternative-drift-formula}, where we derive an alternative formula for the drift $\lambda_\ell$, which will be useful in Section \ref{sec:analyticity}. 

\begin{proof}[Proof of Theorem \ref{thm:alternative-drift-formula}]
Existence of $\lambda_\ell$ was already shown in Theorem \ref{thm:drift}.\\
Recall that, for $g=g_1t_1\dots g_kt_kg_{k+1}$ in normal form, we write $[g]:=g_1t_1\dots g_kt_k$. We set $\mathbb{E}_\pi\bigl[\ell([X_{\e{2}}])-\ell([X_{\e{1}}])\bigr]$ as
\begin{eqnarray*}
&&\sum_{\substack{x=(w_1t_1h_1,m_1),\\y=(w_2t_2h_2,m_2)\in\mathbb{D}}} \pi( x)\cdot q\bigl(x,y\bigr)\cdot \bigl(\ell(w_1t_1w_2t_2)-\ell(w_1t_1) \bigr)\\
&=& \sum_{(w_2t_2h_2,m_2)\in\mathbb{D}}\pi( w_2t_2h_2,m_2)\cdot \ell(w_2t_2)
\end{eqnarray*}
and
\begin{eqnarray*}
\mathbb{E}_\pi[\e{2}-\e{1}]
&:= &\sum_{\substack{(w_1t_1h_1,m_1),\\ (w_2t_2h_2,m_2)\in\mathbb{D}}} \pi( w_1t_1h_1,m_1)\cdot q\bigl((w_1t_1h_1,m_1),(w_2t_2h_2,m_2)\bigr)\cdot m_2\\
&=& \sum_{(w_2t_2h_2,m_2)\in\mathbb{D}}\pi( w_2t_2h_2,m_2)\cdot m_2.
\end{eqnarray*}
That is, we take the expectations w.r.t. the invariant measure of the positive recurrent Markov chain $\bigl((\mathbf{W}_k,\mathbf{i}_k),(\mathbf{W}_{k+1},\mathbf{i}_{k+1})\bigr)_{k\in\N}$. Finiteness of both expectations follows from Lemma \ref{lem:sum-finite} together with at most polynomial growth of $\ell$. 
\par
By the ergodic theorem for positive recurrent Markov chains, we obtain
$$
\frac{1}{\mathbf{k}(n)}\sum_{i=1}^{\mathbf{k}(n)}\bigl(\ell([X_{\e{i}}])-\ell([X_{\e{i-1}}])\bigr)\xrightarrow{n\to\infty}\mathbb{E}_\pi\bigl[\ell([X_{\e{2}}])-\ell([X_{\e{1}}])\bigr] \ \textrm{almost surely.}
$$
Furthermore, we observe that
$$
\frac{1}{\mathbf{k}(n)}\sum_{j=2}^{\mathbf{k}(n)} \e{j}-\e{j-1}=\frac{1}{\mathbf{k}(n)}\sum_{j=2}^{\mathbf{k}(n)}  \mathbf{i}_j\xrightarrow{n\to\infty}\mathbb{E}_\pi[\mathbf{i}_2]=\mathbb{E}_\pi[\e{2}-\e{1}] \ \textrm{almost surely.}.
$$
Hence, 
$$
\frac{\e{\mathbf{k}(n)}}{\mathbf{k}(n)}\xrightarrow{n\to\infty} \mathbb{E}_\pi[\e{2}-\e{1}]=\Lambda \ \textrm{almost surely.}
$$
Since
$$
0\leq \frac{n-\e{\mathbf{k}(n)}}{\mathbf{k}(n)} \leq \frac{\e{\mathbf{k}(n)+1}-\e{\mathbf{k}(n)}}{\mathbf{k}(n)}
\xrightarrow{n\to\infty} 0 \ \textrm{almost surely},
$$
we get
$$
\frac{n}{\mathbf{k}(n)} = \frac{n-\e{\mathbf{k}(n)}}{\mathbf{k}(n)}+ \frac{\e{\mathbf{k}(n)}}{\mathbf{k}(n)} \xrightarrow{n\to\infty} \mathbb{E}_\pi[\e{2}-\e{1}] \ \textrm{almost surely.}
$$
From the proof of Theorem \ref{thm:drift} follows now the claim:
\begin{eqnarray*}
\lambda_\ell &=& \lim_{n\to\infty} \frac{\ell(X_{\e{\mathbf{k}(n)}})}{n}
= \lim_{n\to\infty} \frac{\mathbf{k}(n)}{n} \frac{1}{\mathbf{k}(n)}\sum_{i=1}^{\mathbf{k}(n)} \bigl(\ell([X_{\e{i}}])-\ell([X_{\e{i-1}}])\bigr)\\
&=& \frac{\mathbb{E}_\pi[\ell([X_{\e{2}}])-\ell([X_{\e{1}}])]}{\mathbb{E}_\pi[\e{2}-\e{1}]}\quad \textrm{almost surely.}
\end{eqnarray*}
\end{proof}

\begin{remark}\normalfont
The required condition of a length function $\ell$ of at most polynomial growth can be relaxed to the condition that
$$
\sum_{(w_0t_0h_0,n_0)\in\mathbb{D}}\max\bigl\lbrace \ell(w_0t_0),n\bigr\rbrace\cdot \pi(w_0t_0h_0,n_0)<\infty.
$$
However, this condition is in general hard to prove, because it needs good knowledge of $\pi$. Nonetheless, we may allow word length functions of the following form: let be $\varrho\in \bigl(1,R(\mathcal{K})\bigr)$, where $R(\mathcal{K})$ is the radius of convergence of $\mathcal{K}(z)$; assume that $\ell$ satisfies $\ell(g_0)\leq C\cdot \varrho^{|g_0|}$ for all $g_0\in G_0$. Then one can show analogously to Lemma \ref{lem:sum-finite} that 
$$
\sum_{(w_0t_0h_0,n_0)\in\mathbb{D}} \pi(w_0t_0h_0,n_0)\cdot \ell(w_0t_0)<\infty.
$$
Once again, $R(\mathcal{K})$ is hard to determine, so we restricted the proofs to a general class of meaningful length functions.
\end{remark}

As an application we derive an upper bound for the random walk's \textit{entropy}, which is given by the non-negative constant $h$ such that
$$
h=\lim_{n\to\infty} -\frac{1}{n}\log \pi_n(X_n) \quad \textrm{ almost surely},
$$
where $\pi_n$ is the distribution of $X_n$. Again, existence of the entropy  is well-known due to Kingman's subadditive ergodic theorem. 
\par
For $g\in G$, define \mbox{$F(e,g):=\P[\exists n\in \N: X_n=g]$.} We choose now the \textit{Greenian distance} as length function, that is, 
$$
\ell(g):=\ell_G(g):=-\log F(e,g)\quad \textrm{ for  } g\in G_0\cup\{t,t^{-1}\};
$$ 
compare with  \cite{blachere-haissinsky-mathieu}. If the minimal single step transition probability is given by \mbox{$\varepsilon_0:=\min\{p(e,g)\mid g\in G, p(e,g)>0 \}$,} then
$$
\ell_G(g)=-\log F(e,g)\leq -\log \varepsilon_0^{|g|} =-|g|\log\varepsilon_0,
$$ 
that is, $\ell$ is of polynomial growth, and therefore $\lambda_{\ell_G}$ exists due to Theorem \ref{thm:alternative-drift-formula}. Moreover, we get a simple upper bound for the entropy:
\begin{corollary}\label{cor:entropy}
$\lambda_{\ell_G} \geq h$.
\end{corollary}
\begin{proof}
By  \cite{benjamini-peres94}, the asymptotic entropy can be rewritten as 
\begin{equation}\label{equ:entropy}
h=\lim_{n\to\infty} -\frac{1}{n}\log G(e,X_n|1).
\end{equation}
For $m,n\in\N$, $m<n$, $x_1,\dots,x_m,x\in G_0$, we have 
$$%\begin{equation*}
\P[X_n=x]\geq \P\bigl[\exists k_1<k_2<\ldots < k_m<n: X_{k_1}=x_1,\dots,X_{k_m}=x_m,X_n=x\bigr].
$$%\end{equation*}
By conditioning on the first visits to $x_1,\dots,x_m,x$ we obtain due to vertex transitivity:
\begin{eqnarray}
G(e,x)&\geq & F(e,x_1)\cdot F(x_1,x_2)\cdot\ldots \cdot F(x_m,x)\label{equ:bound-path} \\
&=&F(e,x_1)\cdot F(e,x_1^{-1}x_2)\cdot\ldots \cdot F(e,x_m^{-1}x).\nonumber
\end{eqnarray}
%
%
%Due to, e.g., \cite[Lemma 1.13.(b)]{woess} and vertex transitivity we have
%$$
%G(e,X_n)=F(e,X_n)\cdot G(X_n,X_n)=F(e,X_n)\cdot G(e,e),
%$$
%which in turn yields together with (\ref{equ:entropy}):
%$$
%h=\lim_{n\to\infty} -\frac{1}{n}\log F(e,X_n).
%$$
%Moreover, the triangle inequality and  vertex-transitivity yield
If $X_{\e{i-1}}=g_1t_1\dots g_{i-1}t_{i-1}h_{i-1}$ and $X_{\e{i}}=g_1t_1\dots g_it_ih_i$ are in normal form, then $X_{\e{i-1}}^{-1}X_{\e{i}}=h_{i-1}^{-1}g_it_ih_i=h_{i-1}^{-1}g_i\varphi^\delta(h_i)t_i$, where $\delta=1$, if $t_i=t^{-1}$, and $\delta=-1$, if $t_i=t$.
Therefore, setting $X_{\e{0}}:=e$, we may apply the inequality (\ref{equ:bound-path}) twice, which yields
$$
G(e,X_{\e{n}}) \geq  \prod_{i=1}^n F(e,X_{\e{i-1}}^{-1}X_{\e{i}}) \geq
\prod_{i=1}^n F(e,X_{\e{i-1}}^{-1}X_{\e{i}}t_i^{-1})F(e,t_i).
$$
We obtain the proposed upper bound for $h$ as follows:
\begin{eqnarray*}
h &=& \lim_{n\to\infty} -\frac{1}{\e{n}}\log G(e,X_{\e{n}}) \leq  \lim_{n\to\infty} -\frac{1}{\e{n}} \log \prod_{i=1}^n F(e,X_{\e{i-1}}^{-1}X_{\e{i}}) \\
&\leq & \lim_{n\to\infty} -\frac{1}{\e{n}} \sum_{i=1}^n  \log\bigr[  F(e,X_{\e{i-1}}X_{\e{i}}t_i^{-1})\cdot F(e,t_i)\bigr] \\
&=& \lim_{n\to\infty}\frac{1}{\e{n}}\sum_{i=1}^n \bigl(\ell_G(\underbrace{X_{\e{i-1}}X_{\e{i}}t_i^{-1}}_{\in G_0})+\ell_G(t_i)\bigr) = \lim_{n\to\infty} \frac{1}{\e{n}}\ell_G(X_{\e{n}})=\lambda_{\ell_G}.
\end{eqnarray*}
\end{proof}

\begin{remark}\label{rem:remarks2}\normalfont ~\\
At the end of this section let us discuss why it is considerably more difficult to study the rate of escape w.r.t. the natural graph metric and why the reasoning above can \textit{not} be applied straight-forwardly. It is unclear under which (natural) conditions the graph metric can be expressed by length functions. 
%In general, the rate of escape w.r.t. the natural graph metric may \textit{not} necessarily be expressed by length functions as introduced above. 
This is due to the fact that shortest paths in HNN extensions may follow a subtle behaviour, which seems to be quite cryptic how  to cut shortest paths into i.i.d. pieces, which stabilize as $n\to\infty$. In order to give an idea of the obstacles consider a group $G_0$ with finite isomorphic subgroups $A,B\subsetneq G_0$ such that $A\cap B\neq \{e\}$ and $\varphi(A\cap B)=A\cap B$. Take any $a\in A\cap B$, $a\neq e$, and suppose that $\mu_0(a)>0$. For $n\in\N$, a shortest path (i.e., a sequence of vertices $(v_0,v_1,\dots,v_m)\in G^{m+1}$ with $\mu(v_{i-1}^{-1}v_i)>0$ and $m$ minimal) from $e$ to $g:=t^n\varphi^n(a)$ is given by 
$$
\Pi_1=\bigl(e,a,t\varphi(a),t^2\varphi^2(a),\dots, t^n\varphi^n(a)\bigr);
$$ 
this path has length $n+1$. Note that $d\bigl(e,\varphi^n(a)\bigr)$ could be large.
Moreover, the unique shortest path from $e$ to $t^n$ is given by $\Pi_2=(e,t,t^2,\dots,t^n)$, a path of \mbox{length $n$.} Thus, if the random walk stands at time $k$ at $X_k=t^n\varphi^n(a)$ and at some time $l>k$ at $X_l=t^n$, then the path $\Pi_1$, which is a shortest path from $e$ to $X_k$, has to be changed at all points in order to transform it into the path $\Pi_2$, which is now a shortest path from $e$ to $X_l$. In other words, in this situation no initial part of a shortest path from $e$ to $X_n$, $n>k$, may have stabilized yet.
\par
Note also that a shortest path from $e$ to $a\in A$ could be $(e,t,t\varphi(a)=at,a)$, that is, shortest paths to elements in $G_0$ could make abbreviations through the ``exterior'' of $G_0$.
\par
It is unclear if and how paths can be chosen such that initial parts stabilize. Further deeper investigation is needed in order to understand the behaviour of shortest paths from $e$ to $X_n$ as $n\to\infty$, requiring a different approach which would go beyond the scope of this article.
\end{remark}

\section{Central Limit Theorem}
\label{sec:clt}
In this section we derive a central limit theorem for the word length w.r.t. the length function $\ell$. We still assume that $\ell$ has at most polynomial growth and satisfies \mbox{$\ell(g_0)\leq C\cdot |g_0|^\kappa$} for some $\kappa\in\N$ and all $g_0\in G_0$.
Before we are able to prove Theorem \ref{thm:clt} we have to introduce further notation. Observe that \mbox{$s_0:=(e_0te_0,1)\in\mathbb{D}$} is a state, which can be taken by the Markov chain $(\W_k,\i{k})_{k\in\N}$ with positive probability. Define $\tau_0:=\inf\{ m\in\N \mid (\W_m,\i{m})=s_0\}$ and inductively for $k\geq 1$
$$
\tau_k:=\inf\bigl\lbrace  m >\tau_{k-1} \,\bigl|\,  (\W_m,\i{m})=s_0\bigr\rbrace.
$$
Positive recurrence of $(\W_k,\i{k})_{k\in\N}$ yields $\tau_k<\infty$ almost surely for all $k\in\N$. Furthermore, we define for $i\in\N_0$:
\begin{eqnarray}\label{def:regeneration-times}
T_i := \mathbf{e}_{\tau_i}.
\end{eqnarray}
The following two lemmas contain the keys for later proofs.
\begin{lemma}\label{lem:tau-expmom}
The random variable $\tau_1-\tau_0$ has exponential moments, that is, there is a constant $c_\tau>0$ such that $\mathbb{E}\bigr[\exp\bigr(c_\tau (\tau_1-\tau_0)\bigr)\bigr]<\infty$.
\end{lemma}
\begin{proof}
We will just prove the lemma for the case $A,B\subsetneq G_0$; the remaining case of $A=B=G_0$ with $p\neq \frac12$ is outsourced to Lemma \ref{lem:tau-expmom-A=G} in the Appendix.
\par
For every state $(g_0t_0h_0,n_0)\in\mathbb{D}$ of $(\W_k,\i{k})_{k\in\N}$, the probability of reaching $(e_0te_0,1)$ in two steps is strictly positive:  assume $A,B\subsetneq G_0$ and let be $x\in X\setminus\{e_0\}$ and $n_{h_0}\in\N$ with $\mu_0^{(n_{h_0})}(h_0^{-1}x)>0$; then
\begin{eqnarray*}
q\bigl((g_0t_0h_0,n_0),(xte_0,n_{h_0}+1) \bigr)&\geq& \frac{\xi(te_0)}{\xi(t_0h_0)}\cdot \alpha^{n_{h_0}}\cdot \mu_0^{(n_{h_0})}(h_0^{-1}x)\cdot (1-\alpha) \cdot p>0,\\
q\bigl((xte_0,n_{h_0}+1),(e_0te_0,1)\bigr) &\geq  & \frac{\xi(te_0)}{\xi(te_0)}\cdot (1-\alpha)\cdot p>0,
\end{eqnarray*}
which provides
$$
q:=\min_{h_0\in A\cup B}q\bigl((g_0t_0h_0,n_0),(xte_0,n_{h_0}+1)\bigr)\cdot q\bigl((xte_0,n_{h_0}+1),(e_0te_0,1)\bigr)>0.
$$
%Since there are infinitely many such $n_1\in\N$ with $\mu_0^{(n_1)}(h_0^{-1}x)>0$ due to $\mathrm{supp}(\mu_0)=G_0$, we get 
%For any $n_0\in\N$ with $\mu_0^{(n_0)}(e_0)>0$, we have $q\bigl((xte_0,n_1+1),(e_0te_0,n_0+1)\bigl)>0$, which shows together with $\mathrm{supp}(\mu_0)=G_0$ that 
%$q<1$. 
This leads to the following exponential decaying upper bound:
$$
\mathbb{P}[\tau_1-\tau_0=n]\leq (1-q)^{\lfloor \frac{n}{2}\rfloor},
$$
that is, the random variable $\tau_1-\tau_0$ has exponential moments.
\end{proof}
Furthermore, we can also show:
\begin{lemma}\label{lem:T-expmom}
The random variables $T_0$ and $T_1-T_0$ have exponential moments, that is, there are  constants $c_0>0$ and $c_1>0$ such that $\mathbb{E}\bigr[\exp\bigr(c_0 T_0\bigr)\bigr]<\infty$ and $\mathbb{E}\bigr[\exp\bigr(c_1 (T_1-T_0)\bigr)\bigr]<\infty$.
\end{lemma}
\begin{proof}
Once again we only consider the case $A,B\neq G_0$; the remaining case \mbox{$A=B=G_0$} with $p\neq\frac12$ works similarly, see Lemma \ref{lem:tau-expmom-A=G}.
\par
Let be $x\in X\setminus\{e_0\}$.
Similarly as in the proof of Lemma \ref{lem:tau-expmom}, at any time \mbox{$n\in [T_0,T_{1})$} the random walk $(X_n)_{n\in\N_0}$ can realize the  time $T_{1}$ within the next \mbox{$N:=\max\{n_h\mid h\in A\cup B\}+2$} steps, where \mbox{$n_h:=\min\{ m\in\N\mid \mu_0^{(m)}(h^{-1}x)>0\}$:} if $X_n=g_1t_1\dots g_jt_jg_{j+1}$ (in the form of (\ref{equ:normalform})) then one can walk inside the set of words having prefix $[X_n]$ via $g_1t_1\dots g_jt_jx$ and $g_1t_1\dots g_jt_jxt$ to $g_1t_1\dots g_jt_jxtt$, where $T_{1}$ can be generated.
Hence,  there is some $q_T\in (0,1)$ such that 
$$
\mathbb{P}[T_1-T_0=n]\leq (1-q_T)^{\lfloor \frac{n}{N}\rfloor},
$$
which yields existence of exponential moments of $T_1-T_0$. The same reasoning shows existence of exponential moments of $T_0$.
\end{proof}

Assume now that $(X_n)_{n\in\N_0}$ tends to some $g_1t_1g_2t_2\ldots\in\mathcal{B}$ in the sense of Proposition \ref{lem:infinite-words}.  For $i\in\mathbb{N}$, we define:
\begin{eqnarray}
\widetilde{L}_i&:=& \sum_{j=\tau_{i-1}+1}^{\tau_i} \ell(g_jt_j)=\ell([X_{T_i}])-\ell([X_{T_{i-1}}])\nonumber\\
\textrm{and }\quad L_i&:=&\widetilde{L}_i-(T_i-T_{i-1})\cdot\lambda_\ell.\label{def:Li}
\end{eqnarray}

\begin{lemma}\label{lem:L-variance}
$$
\sigma_L^2:=\mathrm{Var}(L_1)\in(0,\infty).
$$
\end{lemma}
\begin{proof}

Since
\begin{eqnarray*}
\widetilde{L}_1&=& \sum_{j=\tau_{0}+1}^{\tau_1} \ell(g_jt_j) \leq C\cdot \sum_{j=\tau_0+1}^{\tau_1}\mathbf{i}_{j}^\kappa+ \max\{\ell(t),\ell(t^{-1})\}\cdot (\tau_1-\tau_0)\\
&\leq& C\cdot (T_1-T_0)^\kappa +  \max\{\ell(t),\ell(t^{-1})\}\cdot (\tau_1-\tau_0),
\end{eqnarray*}
finiteness of $\sigma_L^2$ follows from Lemmas \ref{lem:tau-expmom} and \ref{lem:T-expmom}. Since $\mu_0$ generates $G_0$ as a semigroup, the random walk can perform arbitrarily many circles in a copy of $G_0$ (in the underlying Cayley graph) in the time interval $[T_0,T_{1}]$; therefore, $L_1$ is \textit{not} constant, and consequently we obtain \mbox{$\sigma_L^2>0$.}
%Assume that $X_{T_0}=t=e_0te_0$; then the random walk can escape to infinity through $X_{T_1}=tt=e_0te_0te_0$ or also through \mbox{$X_{T_1}=tyt^{-1}xtt=tyt^{-1}xte_0te_0$,} where $x\in X\setminus\{e_0\}$,  $y\in Y\setminus\{e_0\}$with $\ell(x)>0$ or $\ell(y)>0$ if $\ell(t)=0$. That is, $L_1$ is \textit{not} constant, providing \mbox{$\sigma_L^2>0$.}
\end{proof}
Completely analogously to Theorem \ref{thm:alternative-drift-formula} one can prove that
\begin{equation}\label{equ:lambda-T-formula}
\lambda_{\ell}=\frac{\mathbb{E}\bigl[\ell([X_{T_1}])-\ell([X_{T_0}])\bigr]}{\mathbb{E}[T_1-T_0]}=\frac{\mathbb{E}[\widetilde{L}_1]}{\mathbb{E}[T_1-T_0]}.
\end{equation}
Observe that we may take expectations w.r.t. the underlying probability measure induced from $\mu$ (that is, w.r.t. the initial distribution $\mathbb{P}[\mathbf{W}_1=\cdot,\mathbf{i}_1=\cdot]$), and \textit{not} w.r.t. the invariant probability measure $\pi$ as initial distribution; this is possible since the random times $T_i$ are regeneration times.
\begin{corollary}\label{cor:VarL1-formula}
$$
\mathbb{E}[L_1]=0 \ \textrm{ and } \ \sigma_L^2= \mathbb{E}\bigl[\bigl(\ell([X_{T_1}])-\ell([X_{T_0}])-(T_1-T_0)\lambda_{\ell}\bigr)^2\bigr].
$$
\end{corollary}
\begin{proof}
We obtain $\mathbb{E}[L_1]=0$ immediately from (\ref{equ:lambda-T-formula}), and therefore the proposed formula for $\sigma_L^2$.
%
%The proposed formula follows immediately from  (\ref{equ:lambda-T-formula}):
%$$
%\mathrm{Var}(L_1)= \mathbb{E}\Bigl[ \bigl(L_1-\mathbb{E}(L_1)\bigr)^2\Bigr] 
%= \mathbb{E}\Bigl[ \bigl(\ell(X_{T_1})-\ell(X_{T_0})-\mathbb{E}[T_1-T_0]\cdot \lambda_{\ell}\bigr)^2\Bigr].
%$$
\end{proof}

Now we can prove the proposed central limit theorem, where we use a similar reasoning as in \cite[Thm. 1.1]{haissinsky-mathieu-mueller12}:

%\begin{comment}
%Key idea for central limit theorem:
%\begin{itemize}
%\item Take some state $s_0$ of the positive-recurrent exit time Markov chain and consider the stopping times $T_1,T_2,\dots$ of the infinitely many visits of $s_0$.
%\item Let $\xi_i$ be the additional added length between $T_{i-1}$ and $T_i$. Then $(\xi_i)$ is i.i.d. and we can apply CLT.
%\item It remains to control what happens between $n$ and $T_{k(n)}$. we have to show that the added length between $n$ and $T_{k(n)}$ divided by $n$ tends to $0$ with probability $1$. This is done with Chebyshev's Inequality. The rest follows from Lemma of Slutsky.
%\end{itemize}
%\end{comment}

\begin{proof}[Proof of Theorem \ref{thm:clt}]
First, observe that Corollary \ref{cor:VarL1-formula} together with Lemma \ref{lem:L-variance} ensures that $\sigma^2$ as defined in Theorem \ref{thm:clt} is strictly positive.
\par
For $k\in\mathbb{N}$,  set
$$
R_k:=\sum_{i=1}^k \widetilde{L}_i, \quad 
S_k:=\sum_{i=1}^k L_i  = R_k-(T_k-T_0)\cdot\lambda_{\ell},
$$
and, for $n\in\N$, set 
$$
\mathbf{t}(n):=\sup\{m\in\N_0 \mid T_m\leq n\}.
$$ 
%If $\mathbf{t}(n)=-\infty$, then we set $T_{-\infty}:=0$.
We note that $\mathbf{t}(n)\to\infty$ almost surely as $n\to\infty$.
Observe that Proposition \ref{lem:exit-time-chain} immediately implies that $(\tau_i-\tau_{i-1})_{i\in\N}$ and $(T_i-T_{i-1})_{i\in\N}$ are  i.i.d. sequences.
The sequence $(L_i)_{i\in\N}$ is also an i.i.d. sequence of random variables; for a proof, we refer to \mbox{Lemma \ref{lem:Li-iid}} in the Appendix. 
%This yields together with (\ref{equ:lambda-T-formula}) that
%$$
%\mathbb{E}[S_k]=\sum_{i=1}^k \mathbb{E}\bigl[L_1-(T_1-T_{0})\cdot \lambda_{\ell} \bigr] = k\cdot \mathbb{E}[L_1]-k\cdot \mathbb{E}[T_1-T_{0}]\cdot \lambda_{\ell}=0.
%$$
Then, by  \cite[Theorem 14.4]{billingsley:99},
$$
\frac{S_{\mathbf{t}(n)}}{\sigma_L \sqrt{\mathbf{t}(n)}} \xrightarrow{\mathcal{D}} N(0,1).
$$
Analogously as in the proof of Theorem \ref{thm:alternative-drift-formula}, one can show that
$$
\frac{n}{\mathbf{t}(n)}\xrightarrow{n\to\infty} \mathbb{E}[T_1-T_0]\quad \textrm{almost surely.}
$$
Applying the Lemma of Slutsky gives for $n$ large enough:
\begin{equation}\label{equ:clt}
\frac{S_{\mathbf{t}(n)}}{\sigma_L \sqrt{n}} = \frac{S_{\mathbf{t}(n)}}{\sigma_L \sqrt{\mathbf{t}(n)}}\frac{\sqrt{\mathbf{t}(n)}}{\sqrt{n}}\xrightarrow{\mathcal{D}} N\biggl(0,\frac{1}{\mathbb{E}[T_1-T_0]}\biggr).
\end{equation}
The next step is to prove the following convergence behaviour:
\begin{equation}\label{equ:overhead}
\frac{\ell(X_n)- R_{\mathbf{t}(n)}}{ \sqrt{n}}\xrightarrow{\P} 0.
\end{equation}
Assume that $\mathbf{t}(n)\geq 1$ and that $X_n$ has the form 
$$
g_1t_1\ldots g_{\tau_{\mathbf{t}(n)}}t_{\tau_{\mathbf{t}(n)}}g_{\tau_{\mathbf{t}(n)}+1}t_{\tau_{\mathbf{t}(n)}+1} \ldots g_{m}t_{m}g_{m+1},
$$ 
where $m=|X_n|_t$. Recall that $R_{\mathbf{t}(n)}$ does not contain the weights of the letters of $X_{T_0}$.
Polynomial growth of $\ell$ yields the following upper bound:
\begin{eqnarray*}
&&\ell(X_n)-R_{\mathbf{t}(n)} \\
&= &  \sum_{j=\tau_{\mathbf{t}(n)}+1}^{m} \bigl(\ell(g_j)+\ell(t_j)\bigr) + \ell(g_{m+1}) + \ell(X_{T_0}) \\
&\leq &  \sum_{j=\tau_{\mathbf{t}(n)}+1}^{m} \bigl(C\cdot |g_j|^{\kappa}+\ell(t_j)\bigr) +C\cdot |g_{m+1}|^\kappa +  \sum_{k=1}^{\tau_0} \bigl(C\cdot |g_k|^{\kappa}+\ell(t_k)\bigr) \\
&\leq & C\cdot (T_{\mathbf{t}(n)+1}-T_{\mathbf{t}(n)})^\kappa + C\cdot T_0^\kappa
+ \max_{s\in\{t,t^{-1}\}} \ell(s)\cdot \bigl( (T_{\mathbf{t}(n)+1}-T_{\mathbf{t}(n)})+T_0\bigr)\\
&\leq & C' \cdot (T_{\mathbf{t}(n)+1}-T_{\mathbf{t}(n)})^\kappa + C' \cdot T_0^\kappa
\end{eqnarray*}
where $C':=2\cdot \bigl(C+  \max_{s\in\{t,t^{-1}\}}\ell(s)\bigr)$. Since \mbox{$(T_{i}-T_{i-1})_{i\in\N}$} is an i.i.d. sequence, we obtain for any $\varepsilon >0$ and $n$ large enough:
\begin{eqnarray*}
&&\P\bigl[\ell(X_n)- R_{\mathbf{t}(n)}> \varepsilon \sqrt{n} , \mathbf{t}(n)\geq 1\bigr] \\
&\leq & \P\bigl[C' \cdot (T_{\mathbf{t}(n)+1}-T_{\mathbf{t}(n)})^\kappa + C' \cdot T_0^\kappa> \varepsilon \sqrt{n}, \mathbf{t}(n)\geq 1 \bigr] \\
&\leq & \P\Bigl[\exists k\in\{1,\dots,n\}: (T_{k+1}- T_{k})^\kappa \geq  \frac{\varepsilon}{2C'} \sqrt{n}\Bigr] + \P\Bigl[  T_0^\kappa\geq   \frac{\varepsilon}{2C'} \sqrt{n}\Bigr] \\
&\leq & n\cdot \P\Bigl[  (T_{1}-T_{0})^{\kappa} \geq  \frac{\varepsilon}{2C'} \sqrt{n}\Bigr] + \P\Bigl[T_0^{\kappa}\geq   \frac{\varepsilon}{2C'} \sqrt{n}\Bigr] \\
&\leq &  n\cdot \P\Bigl[(T_{1}-T_{0})^{3\kappa} \geq  \Bigl(\frac{\varepsilon}{2C'} \sqrt{n}\Bigr)^3\Bigr] + \P\Bigl[ T_0^{\kappa}\geq   \frac{\varepsilon}{2} \sqrt{n}\Bigr] \\
&\leq & n\cdot \frac{\mathbb{E}\bigl[ (T_{1}-T_{0})^{3\kappa} \bigr]}{\bigl(\frac{\varepsilon}{2C'} \sqrt{n}\bigr)^3} + \frac{\mathbb{E}\bigl[T_0^{\kappa} \bigr]}{\frac{\varepsilon}{2} \sqrt{n}} \xrightarrow{n\to\infty} 0.
\end{eqnarray*}

In the last inequality we applied Markov's Inequality and used Lemma \ref{lem:T-expmom} afterwards.
The proposed convergence behaviour in (\ref{equ:overhead}) follows now from the fact that $\mathbf{t}(n)\to\infty$  almost surely as $n\to\infty$.
\par
By construction of the random times $T_i$, $i\in\N_0$, we have
$$
\mathbf{S}_{\mathbf{t}(n)}=\bigl(\ell([X_{\mathbf{t}(n)}])-\ell([X_{T_0}])\bigr)-(T_{\mathbf{t}(n)}-T_0)\cdot \lambda_{\ell}.
$$
For $\varepsilon >0$, we get:
\begin{eqnarray}
&&\P\Bigl[ \bigl| \mathbf{S}_{\mathbf{t}(n)} -\bigl(\ell(X_n)-n\cdot\lambda_{\ell}\bigr)\bigr| >\varepsilon \sqrt{n}\Bigr]\quad \label{equ:inequ1}\\
&\leq & \P\Bigl[ \ell(X_n)-\ell([X_{\mathbf{t}(n)}])+\ell([X_{T_0}]) \geq  \frac{\varepsilon}{2} \sqrt{n}\Bigr] +
\P\Bigl[ \lambda_{\ell}  \bigl(n - (T_{\mathbf{t}(n)}-T_0)\bigr) \geq  \frac{\varepsilon}{2} \sqrt{n}\Bigr].\nonumber
\end{eqnarray}
From (\ref{equ:overhead}) follows that
$$
\frac{\ell(X_n)-\ell([X_{\mathbf{t}(n)}])+\ell([X_{T_0}])}{\sqrt{n}} %= \frac{\ell(X_n)-\sum_{j=1}^{\mathbf{t}(n)} \widetilde{L}_j}{\sqrt{n}}
=\frac{\ell(X_n)-R_{\mathbf{t}(n)}}{\sqrt{n}}\xrightarrow{\P} 0,
$$
hence $\P\Bigl[ \ell(X_n)-\ell([X_{\mathbf{t}(n)}])+\ell([X_{T_0}]) \geq  \frac{\varepsilon}{2} \sqrt{n}\Bigr]\to 0$ as $n\to\infty$.
For the second summand in (\ref{equ:inequ1}), we obtain once again from $T_i-T_{i-1}\sim T_1-T_0$ for all $i\in\N$:
\begin{eqnarray*}
&& \P\Bigl[ \lambda_{\ell} \cdot \bigl(n - (T_{\mathbf{t}(n)}-T_0)\bigr) \geq  \frac{\varepsilon}{2} \sqrt{n},\mathbf{t}(n)\geq 1\Bigr] \\
&\leq & \P\Bigl[ \lambda_{\ell} \cdot \bigl(T_{\mathbf{t}(n)+1} - (T_{\mathbf{t}(n)}-T_0)\bigr) \geq  \frac{\varepsilon}{2} \sqrt{n},\mathbf{t}(n)\geq 1\Bigr]\\
&\leq & \P\Bigl[ \exists k\in\{1,\dots,n\}: T_{k+1}-T_k \geq  \frac{\varepsilon}{4\lambda_{\ell}} \sqrt{n}\Bigr] + \P\Bigl[T_0 \geq  \frac{\varepsilon}{4\lambda_{\ell}} \sqrt{n}\Bigr] \\
%&\leq & n\cdot \P\Bigl[T_{2}-(T_1-T_0) \geq  \frac{\varepsilon}{2\lambda_{\ell}} \sqrt{n}\Bigr]\\
&\leq & n\cdot \P\Bigl[T_{1}-T_0 \geq  \frac{\varepsilon}{4\lambda_{\ell}} \sqrt{n}\Bigr] +  \P\Bigl[T_0 \geq  \frac{\varepsilon}{4\lambda_{\ell}} \sqrt{n}\Bigr] \\ 
&= & n\cdot \P\Bigl[(T_{1}-T_0)^4 \geq  \frac{\varepsilon^4}{(4\lambda_{\ell})^4} n^2\Bigr] + \P\Bigl[T_0 \geq  \frac{\varepsilon}{4\lambda_{\ell}} \sqrt{n}\Bigr] \\ 
&\leq &  n\cdot  (4\lambda_{\ell})^4\cdot \frac{\mathbb{E}\bigl[(T_1-T_0)^4\bigr]}{\varepsilon^4 n^2}+
 4\cdot \lambda_{\ell}\cdot  \frac{\mathbb{E}\bigl[T_0\bigr]}{\varepsilon \sqrt{n}} \xrightarrow{n\to\infty} 0.
\end{eqnarray*}
We applied Markov's Inequality in the last line together with Lemma \ref{lem:T-expmom}. As $\mathbf{t}(n)\to\infty$ almost surely, we obtain
$$
\P\bigl[ \bigl| S_{\mathbf{t}(n)} -\bigl(\ell(X_n)-n\cdot\lambda_{\ell}\bigr)\bigr| >\varepsilon \sqrt{n}\bigr] \xrightarrow{\P} 0.
$$
Another application of the Lemma of Slutsky together with (\ref{equ:clt})  proves the claim.
%Since $\mathbb{E}[\mathbf{i}_1^{2K}]<\infty$, we get the required uniform boundness.
%\par
%Finally, we remark that
%$$
%\mathbb{E}[l(X_{T_1})-l(X_{T_0})-(T_1-T_0)\lambda_{\ell}]=0
%$$
%yielding the proposed formula for $\sigma^2$.
\end{proof}

%\begin{remark}
%One can show that 
%$$
%\sigma=\mathbb{E}\bigl[\bigl(\ell(X_{T_1})-\ell(X_{T_0})-(T_1-T_0)\lambda_\ell\bigr)^2\bigr].
%$$
%This can be verified by replacing $S_k$ in the proof of Theorem \ref{thm:clt} by the centered random variables
%$$
%S_k':=\sum_{i=1}^k L_i-(T_i-T_{i-1})\lambda_\ell.
%$$
%\end{remark}

\section{Analyticity of $\lambda_\ell$}
\label{sec:analyticity}

In this section we show that $\lambda_\ell$ varies real-analytically in terms of probability measures of constant support. To this end, we show that both nominator and denominator in the formula for $\lambda_\ell$ given in  (\ref{equ:lambda-T-formula}) vary real-analytically in the parameters describing the random walk on $G$.
\par
First, we describe the problem more formally. Let $S_0=\{s_1,\dots,s_{d}\}$ generate $G_0$ as a semigroup and denote by 
$$
\mathcal{P}_0(S_0)=\Bigl\lbrace (p_1,\dots,p_d) \,\Bigl|\, \forall i\in\{1,\dots,d\}: p_i>0, \sum_{j=1}^d p_j=1\Bigr\rbrace
$$ 
the set of all strictly positive probability measures $\mu_0$ on $S_0$ with 
$$
\bigl(\mu_0(s_1),\dots,\mu_0(s_d)\bigr):=(p_1,\dots,p_d)\in\mathcal{P}(S_0).
$$
Consider the parameter vector 
$$
\underline{p}:=\bigl(p_1,\dots, p_d,\alpha,\beta,p,q)\in \mathcal{P}_0(S)\times (0,1)^4.
$$
The set of valid parameter vectors, whose single entries describe uniquely the random walk probability measure $\mu$ on $G$ is given by
$$
\mathcal{P}:=\mathcal{P}_0(S)\times \bigl\lbrace(\alpha,\beta)\in (0,1)^2\mid \beta=1-\alpha\bigr\rbrace \times
\bigl\lbrace (p,q)\in (0,1)^2\mid q=1-p\bigr\rbrace,
$$
if $A,B\neq G_0$. In the case $A=B=G_0$ we have to exclude the case $p\neq \frac12$ and set
$$
\mathcal{P}:=\mathcal{P}_0(S)\times \bigl\lbrace(\alpha,\beta)\in (0,1)^2\mid \beta=1-\alpha\bigr\rbrace \times
\bigl\lbrace (p,q)\in (0,1)^2\mid q=1-p,p\neq1/2\bigr\rbrace.
$$
Our aim is to show that the mapping
$$
(\mu_0,\alpha,p)\mapsto \lambda_\ell=\lambda_\ell(\mu_0,\alpha,p)
$$
varies real analytically in $(\mu_0,\alpha,1-\alpha,p,1-p)\in\mathcal{P}$, that is, $\lambda_\ell(\mu_0,\alpha,p)$ can be expanded as a multivariate power series in the variables of $\underline{p}$ (with $\beta=1-\alpha$ and $q=1-p$) in a neighbourhood of any $\underline{p}_0\in\mathcal{P}$.

\begin{remark}\normalfont \label{rem:counterexample-analyticity}
At this point let me remark that analyticity of the rate of escape is not obvious: e.g., consider a nearest neighbour random walk $(Z_n)_{n\in\N_0}$ on $\mathbb{Z}$ with transition probabilities 
$$
\mathbb{P}[Z_{n+1}=z+1|Z_n=z]=p_1,\ \mathbb{P}[Z_{n+1}=z-1|Z_n=z]=1-p_1
$$ 
for all $z\in\mathbb{Z},n\in\mathbb{N}$. Then the mapping $(0,1)\ni p_1\mapsto \lambda=|2p_1-1|$ is not analytic. Another counterexample is given in  \cite{mairesse-matheus:2007}.
\end{remark}
We have to give some preliminary remarks, before we present a proof for our analyticity result. Let $A_n$, $n\in\N_0$, be a event which can be described by paths of length $n$ of the Markov chain $(X_n)_{n\in\N_0}$ on $G$; e.g., $A_n=[X_n\in G_0]$. By decomposing each such path belonging to $A_n$ w.r.t. the number of steps which are performed w.r.t. the $d+2$ parameters $\mu(s_i),\mu(t^{\pm1 })$, we can rewrite $\P[A_n]$ as 
\begin{equation}
\sum_{\substack{n_1,\dots,n_{d+2}\geq 0:\\ n_1+\dots+n_{d+2}= n}} c(n_1,\dots,n_{d+2})p_1^{n_1}\cdot
\ldots \cdot p_d^{n_d}\cdot \alpha^{n_1+\dots + n_d}\cdot \beta^{n_{d+1}+n_{d+2}}\cdot p^{n_{d+1}}\cdot q^{n_{d+2}}, \label{equ:path-prob}
\end{equation}
where $c(n_1,\dots,n_{d+2})\in [0,\infty)$.
If the generating function $\mathcal{F}(z):= \sum_{n\geq 0}\P[A_n]\,z^n$, $z\in\mathbb{C}$, has radius of convergence strictly bigger than $1$, then, for $\delta>0$ small enough,
\begin{eqnarray}
\infty >\mathcal{F}(1+\delta)&=&\sum_{n\geq 0}\sum_{\substack{n_1,\dots,n_{d+2}\geq 0:\\ n_1+\dots+n_{d+2}=n}} c(n_1,\dots,n_{d+2}) \prod_{i=1}^d \bigl(\alpha p_i(1+\delta)\bigr)^{n_i}\nonumber\\
&&\quad  \cdot \bigl(\beta p(1+\delta)\bigr)^{n_{d+1}}\cdot \bigl(\beta q(1+\delta)\bigr)^{n_{d+2}};\label{equ:analytic1}
\end{eqnarray}
%Hence, $\underline{p}$ lies in the interior of the domain of convergence of the latter sum, when considered as a multivariate power series in the variables of $\underline{p}=\{p_1,\dots,p_{d+4}\}$;
that is, the mapping $(\mu_0,\alpha,p)\mapsto \mathcal{F}(1)$ varies real-analytically when considered as a power series in $\underline{p}$. 
%This concept shows that, for $\underline{p}\in \mathcal{P}$, the mappings
%\begin{eqnarray*}
%&&\underline{p}\mapsto G(e,g|1), \ g\in G\\
%&&\underline{p}\mapsto \mathcal{K}(1),\\
%&&\underline{p}\mapsto \frac{\partial}{\partial z }\mathcal{K}(z)\bigl|_{z=1},\\
%&&\underline{p}\mapsto \xi(t_1h_1) \quad \forall t_1h_1\in tB\cup t^{-1}A
%\end{eqnarray*}
%are real-analytic functions. Define for $a\in A$, $b\in B$, $z\in\mathbb{C}$:
%$$
%\mathcal{K}_a(z):=\sum_{x\in X}G(e,xa|z),\quad \mathcal{K}_b(z):=\sum_{y\in Y}G(e,yb|z).
%$$
%Then the functions
%$$
%\underline{p}\mapsto\mathcal{K}_a(1):=\sum_{x\in X}G(e,xa|1),\quad \underline{p}\mapsto\mathcal{K}_b(z):=\sum_{y\in Y}G(e,yb|1)
%$$
%vary real-analytically in $\underline{p}\in \mathcal{P}$, since both functions are dominated by $\underline{p}\mapsto \mathcal{K}(1)$; the same holds for 
%$$
%\underline{p}\mapsto \frac{\partial}{\partial z}\bigl[\mathcal{K}_a(z)\bigr]\Bigl|_{z=1},\quad \underline{p}\mapsto \frac{\partial}{\partial z}\bigl[\mathcal{K}_b(z)\bigr]\Bigl|_{z=1},
%$$
%since $\mathcal{K}_a(z), \mathcal{K}_b(z)$ have radius of convergence strictly bigger than $1$.
%
This will be very helpful in the proof of the next two lemmas, which are the essential ingredients for the proof of Theorem \ref{th:analyticity}.
\begin{lemma}\label{lem:ET0}
The mapping
$$
(\mu_0,\alpha,p)\mapsto \mathbb{E}[T_1-T_0]
$$
varies real-analytically.
\end{lemma}
\begin{proof}
First, observe that we can rewrite the expectation as
$$
\mathbb{E}[T_1-T_0]=\sum_{n\geq 1}\mathbb{P}[T_1-T_0=n]\cdot n = \frac{\partial}{\partial z}\biggl[\sum_{n\geq 1}\mathbb{P}[T_1-T_0=n]\cdot z^n\biggr]\Biggl|_{z=1}.
$$
Since $T_1-T_0$ has exponential moments, the power series $\sum_{n\geq 1}\mathbb{P}[T_1-T_0=n]\cdot z^n$ has radius of convergence strictly bigger than $1$. According to the remarks at the beginning of this section it suffices to show that the probabilities $\mathbb{P}[T_1-T_0=n]$, $n\in\N$ can be written in the form of (\ref{equ:path-prob}). We define
$$
\mathbb{D}_{m,n}:=\Bigl\lbrace \bigl( (g_1t_1h_1,n_1), \dots,(g_mt_mh_m,n_m))\in(\mathbb{D}\setminus\{s_0\})^m \,\Bigl|\, n_1+\dots +n_m=n\Bigr\rbrace.
$$
By conditioning on the value of $T_0$ we obtain together with positive recurrence of $(\mathbf{W}_k,\mathbf{i}_k)_{k\in\N}$:
\begin{eqnarray*}
&&\mathbb{P}[T_1-T_0=n] \\
&=& \sum_{k\geq 1}\sum_{w_1,\dots, w_{k-1}\in\mathbb{D}\setminus\{s_0\}} \mathbb{P}\left[\substack{(\mathbf{W}_1,\mathbf{i}_1)=w_1,\dots,(\mathbf{W}_{k-1},\mathbf{i}_{k-1})=w_{k-1}, \\(\mathbf{W}_{k},\mathbf{i}_{k})=s_0}\right] \\
&&\ \cdot \sum_{m=1}^n \sum_{(\bar w_1,\dots, \bar w_{m-1})\in\mathbb{D}_{m-1,n-1}} \mathbb{P}\left[\substack{ (\mathbf{W}_{k+1},\mathbf{i}_{k+1})=\bar w_1,\dots,\\ (\mathbf{W}_{k+m-1},\mathbf{i}_{k+m-1})=\bar w_{m-1}, \\ (\mathbf{W}_{k+m},\mathbf{i}_{k+m})=s_0}\, \biggl| \, (\mathbf{W}_{k},\mathbf{i}_{k})=s_0\right] \\
&=& \sum_{m=1}^n \sum_{(\bar w_1,\dots, \bar w_{m-1})\in\mathbb{D}_{m-1,n-1}} \mathbb{P}\left[\substack{ (\mathbf{W}_{1},\mathbf{i}_{1})=\bar w_1,\dots,(\mathbf{W}_{m-1},\mathbf{i}_{m-1})=\bar w_{m-1},\\ (\mathbf{W}_{m},\mathbf{i}_{m})=s_0} \, \bigl| \, (\mathbf{W}_{0},\mathbf{i}_{0})=s_0\right].
\end{eqnarray*}
Due to the formula in Proposition \ref{lem:exit-time-chain} for the transition probabilities of the process $(\mathbf{W}_k,\mathbf{i}_k)_{k\in\N}$ we can find a set $A_n$, $n\in\N$, of paths of length $n$ of the random walk $(X_n)_{n\in\N_0}$ such that we can rewrite $\mathbb{P}[T_1-T_0=n]$ as 
\begin{equation}\label{equ:Paths-T1-T0}
\mathbb{P}[T_1-T_0=n]=\frac{\xi(te_0)}{\xi(te_0)}\cdot \sum_{\mathrm{Path}\in A_n}\mathbb{P}[\mathrm{Path}]= \sum_{\mathrm{Path}\in A_n}\mathbb{P}[\mathrm{Path}].
\end{equation}
Since every probability $\mathbb{P}[\mathrm{Path}]$, $\mathrm{Path}\in A_n$, can be rewritten in the form of (\ref{equ:path-prob}), we finally get  analyticity of $\mathbb{E}[T_1-T_0]$ as explained in (\ref{equ:analytic1}).
\end{proof}

% OLD:
%
%\begin{lemma}\label{lem:ET0}
%The mapping
%$$
%(\mu_0,\alpha,p)\mapsto \mathbb{E}[T_0]
%$$
%varies real-analytically.
%\end{lemma}
%\begin{proof}
%Since many arguments are very similar to the reasoning in Section \ref{sec:analyticity}, we give only a sketch of the proof. First, observe that we can rewrite the expectation as
%$$
%\mathbb{E}[T_0]=\sum_{n\geq 1}\mathbb{P}[T_0=n]\cdot n = \frac{\partial}{\partial z}\sum_{n\geq 1}\mathbb{P}[T_0=n]\cdot z^n\biggl|_{z=1}.
%$$
%Since $T_1-T_0$ has exponential moments, also $T_0$ has exponential moments, which yields that the power series $\sum_{n\geq 1}\mathbb{P}[T_0=n]\cdot z^n$ has radius of convergence strictly bigger than $1$. Let $A_n$, $n\in\N$, be the set of paths if length $n$ such that at time $n$ the path is at some element ending with the letters $te_0$. Then:
%$$
%\mathbb{P}[T_0=n]=\sum_{\mathrm{Path}\in A_n}\mathbb{P}[\mathrm{Path}]\cdot \xi(te_0).
%$$
%Since every probability $\mathbb{P}[\mathrm{Path}]$, $\mathrm{Path}\in A_n$, can be rewritten in the form of (\ref{equ:path-prob}), we can show analyticity of $\mathbb{E}[T_0]$ as in (\ref{equ:analytic1}).
%\end{proof}
Analogously, we have the following property:
\begin{lemma}\label{lem:EellT0}
The mapping
$$
(\mu_0,\alpha,p)\mapsto \mathbb{E}\bigl[\ell([X_{T_1}])-\ell([X_{T_0}])\bigr]
$$
varies real-analytically.
\end{lemma}
\begin{proof}
We start expanding the expectation $\mathbb{E}\bigl[z^{T_1-T_0}\ell([X_{T_1}])-\ell([X_{T_0}])\bigr]$, $z\in\mathbb{C}$, where we will use the notation $\bar w_k=(g_kt_kh_k,n_k)$ for  $\bar w_k\in\mathbb{D}$:
\begin{eqnarray*}
&&\mathbb{E}\bigl[z^{T_1-T_0}\bigl(\ell([X_{T_1}])-\ell([X_{T_0}])\bigr)\bigr] \\
&=&
\underbrace{\sum_{k\geq 1}\sum_{w_1,\dots, w_{k-1}\in\mathbb{D}\setminus\{s_0\}} \mathbb{P}\left[\substack{(\mathbf{W}_1,\mathbf{i}_1)=w_1,\dots,(\mathbf{W}_{k-1},\mathbf{i}_{k-1})=w_{k-1}, \\(\mathbf{W}_{k},\mathbf{i}_{k})=s_0}\right]}_{=\P[T_0<\infty]=1} \\
&&\ \cdot \sum_{n\geq 1}\sum_{m=1}^n \sum_{(\bar w_1,\dots, \bar w_{m-1})\in\mathbb{D}_{m-1,n-1}} \mathbb{P}\left[\substack{ (\mathbf{W}_{k+1},\mathbf{i}_{k+1})=\bar w_1,\dots, \\(\mathbf{W}_{k+m-1},\mathbf{i}_{k+m-1})=\bar w_{m-1}, \\ (\mathbf{W}_{k+m},\mathbf{i}_{k+m})=s_0}\, \biggl| \, (\mathbf{W}_{k},\mathbf{i}_{k})=s_0\right] \\
&&\ \cdot z^{n_1+\dots + n_{m-1}+1} \cdot \biggl(\sum_{j=1}^{m-1} \ell(g_jt_j)+\ell(e_0t)\biggr)\\
&=& \sum_{n\geq 1} \underbrace{\sum_{m=1}^n \sum_{(\bar w_1,\dots, \bar w_{m-1})\in\mathbb{D}_{m-1,n-1}} \mathbb{P}\left[\substack{ (\mathbf{W}_{1},\mathbf{i}_{1})=\bar w_1,\dots,\\ (\mathbf{W}_{m-1},\mathbf{i}_{m-1})=\bar w_{m-1},\\ (\mathbf{W}_{m},\mathbf{i}_{m})=s_0} \, \biggl| \, (\mathbf{W}_{0},\mathbf{i}_{0})=s_0\right]}_{=\P[T_1-T_0=n]}\\
&&\quad
\cdot  z^{n}\cdot \biggl(\sum_{j=1}^{m-1} \ell(g_jt_j)+\ell(e_0t)\biggr).
\end{eqnarray*}
For real $z>0$, we can bound this sum from above by
$$
\mathbb{E}\bigl[z^{T_1-T_0}\bigl(\ell([X_{T_1}])-\ell([X_{T_0}])\bigr)\bigr]  \leq
  \sum_{n\geq 1} \mathbb{P}[T_1-T_0=n]\cdot z^n \cdot \bigl( C\cdot n^\kappa + n\cdot \max\{ \ell(t),\ell(t^{-1})\}\bigr).
$$
Since the power series $\sum_{n\geq 1}\mathbb{P}[T_1-T_0=n]\cdot z^n$ has radius of convergence strictly bigger than $1$ due to existence of exponential moments of $T_1-T_0$ (see \mbox{Lemma \ref{lem:T-expmom}),} the left hand side of the above inequality converges for  $z=1+\delta$ with $\delta>0$ sufficiently small. Rewriting the left hand side yields
\begin{eqnarray*}
&& \mathbb{E}\bigl[z^{T_1-T_0}\bigl(\ell([X_{T_1}])-\ell([X_{T_0}])\bigr)\bigr] \\
&=& \sum_{n\in\N}z^n \cdot \underbrace{\sum_{s\in  \mathrm{supp}\bigl(\ell([X_{T_1}])-\ell([X_{T_0}])\bigr)} s\cdot \P\bigl[ T_1-T_0=n,\ell([X_{T_1}])-\ell([X_{T_0}])=s\bigr]}_{=:a_n}.
\end{eqnarray*}
For each $n\in\N$ and each $s\in  \mathrm{supp}\bigl(\ell([X_{T_1}])-\ell([X_{T_0}])\bigr)$, we can find -- analogously to (\ref{equ:Paths-T1-T0}) -- a set of paths $A_{n,s}$ of length $n$ such that
$$
a_n=\sum_{s\in  \mathrm{supp}\bigl(\ell([X_{T_1}])-\ell([X_{T_0}])\bigr)} \P[A_{n,s}]\cdot s,
$$
that is, we can write $a_n$ in the form of (\ref{equ:path-prob}). The rest follows as explained in (\ref{equ:analytic1}), which proves analyticity of $\mathbb{E}\bigl[\ell([X_{T_1}])-\ell([X_{T_0}])\bigr]$.
%the power series
%$$
%\mathbb{E}\bigl[z^{T_1-T_0}\bigl(\ell([X_{T_1}])-\ell([X_{T_0}])\bigr)\bigr] =\sum_{\substack{n\in\N,\\ s\in \mathrm{supp}(\ell([X_{T_1}])-\ell([X_{T_0}]))}}\P\bigl[ T_1-T_0=n,\ell([X_{T_1}])-\ell([X_{T_0}])=s\bigr] \cdot 
%$$
%
%
%Analyticity of $\mathbb{E}\bigl[\ell([X_{T_1}])-\ell([X_{T_0}])\bigr]$ follows now from the fact that the power series $\sum_{n\geq 1}\mathbb{P}[T_1-T_0=n]\cdot z^n$ has radius of convergence strictly bigger than $1$ due to existence of exponential moments of $T_1-T_0$, see Lemma \ref{lem:T-expmom}.
\end{proof}

\begin{proof}[Proof of Theorem \ref{th:analyticity}]
The proof follows now directly from Lemmas \ref{lem:ET0} and \ref{lem:EellT0} in view of the drift formula given in (\ref{equ:lambda-T-formula}).
\end{proof}

\begin{proof}[Proof of Theorem \ref{thm:variance-analytic}]
This can be checked analogously to Lemmas \ref{lem:ET0} and \ref{lem:EellT0} with a similar reasoning (without needing any further additional techniques/ideas)  due to existence of exponential moments of $T_1-T_0$.  Therefore, we omit a further, detailed proof at this point. 
\end{proof}

\begin{appendix}
\label{sec:appendix}
\section{Remaining proofs}

\begin{lemma}\label{lem:p-neq-1/2}
Consider the case $A=B= G_0$ and $p\neq \frac12$. Then $G(e,e|z)$ has  radius of convergence strictly bigger than $1$.
\end{lemma}
\begin{proof}
The idea is to trace back this case to a non-symmetric nearest neighbour random walk on $\mathbb{Z}$, from which we can derive the required result.
\par
Let $(Z_n)_{n\in\N_0}$ be a random walk on $\mathbb{Z}$ governed by the probability measure $\mu_{\mathbb{Z}}(1)=p,\mu_{\mathbb{Z}}(-1)=1-p$, that is, we have $\mathbb{P}[Z_{n+1}=x+1\mid Z_n=x]=p$ and \mbox{$\mathbb{P}[Z_{n+1}=x-1\mid Z_n=x]=1-p$} for all $n\in\N$, $x\in\mathbb{Z}$. We define the associated first visit generating functions:
\begin{eqnarray*}
F_{\mathbb{Z}}(0,1|z)&:=& \sum_{n\geq 1}\mathbb{P}_0[Z_n=1,\forall m\in\{1,\dots,m-1\}:Z_m\neq 1]\,z^n,\\
F_{\mathbb{Z}}(0,-1|z)&:=& \sum_{n\geq 1}\mathbb{P}_0[Z_n=-1,\forall m\in\{1,\dots,m-1\}:Z_m\neq -1]\,z^n.
\end{eqnarray*}
The first return generating function is given by 
$$
U_{\mathbb{Z}}(z):=\sum_{n\geq 1}\mathbb{P}_0[Z_n=0,\forall m\in\{1,\dots,m-1\}:Z_m\neq 0]\,z^n. 
$$
Conditioning on the first step gives the following system:
\begin{eqnarray*}
F_{\mathbb{Z}}(0,1|z) &=& \mu_{\mathbb{Z}}(1)\cdot z + \mu_{\mathbb{Z}}(-1)\cdot z \cdot F_{\mathbb{Z}}(0,1|z)^2,\\
F_{\mathbb{Z}}(0,-1|z) &=& \mu_{\mathbb{Z}}(-1)\cdot z + \mu_{\mathbb{Z}}(1)\cdot z \cdot F_{\mathbb{Z}}(0,-1|z)^2,\\
U_{\mathbb{Z}}(z) &=& \mu_{\mathbb{Z}}(1)\cdot z\cdot  F_{\mathbb{Z}}(0,-1|z)+ \mu_{\mathbb{Z}}(-1)\cdot z\cdot  F_{\mathbb{Z}}(0,1|z).
\end{eqnarray*}
Solving this system leads to the formula
$$
U_{\mathbb{Z}}(z)=(1-p)\cdot z \cdot  \frac{1 - \sqrt{1 - 4pz^2 + 4p^2z^2}}{2pz} + p\cdot z \cdot  \frac{1 + \sqrt{1 - 4pz^2 + 4p^2z^2}}{2pz}.
$$
Therefore, $U_{\mathbb{Z}}(z)$  has radius of convergence strictly bigger than $1$ and satisfies $U_{\mathbb{Z}}(1)<1$ due to transience, and consequently
$$
G_{\mathbb{Z}}(z):=\sum_{n\geq 0} \mu_{\mathbb{Z}}^{(n)}(0)\cdot z^n=\frac{1}{1-U_{\mathbb{Z}}(z)}
$$
has also radius of convergence strictly bigger than $1$.
\par
We now turn back to our random walk on $G$. Define the stopping times
$$
s(0):=0,\quad \forall k\in\N: s(k):=\min\bigl\lbrace m> s(k-1) \mid X_{m-1}^{-1}X_m\in \{t,t^{-1}\}\bigr\rbrace.
$$
That is, $s(k)$ is the $k$-th time that the random walk on $G$ performs a step w.r.t. $\delta_{t^{\pm 1}}$.
Due to transience and finiteness of $A=B=G_0$, $s(k)<\infty$ almost surely for all $k\in\N$. For $k\geq 1$, $n_0:=0,n_1,\dots,n_k\in\mathbb{Z}$, define
$$
w(n_1,\dots,n_k):=\mathbb{E} \Bigl[ z^{s(k)}\mathds{1}_{[X_{s(j)}\in t^{n_j}G_0 \forall j\in\{1,\dots,k\}]}\,\Bigl|\, X_0=e\Bigr].
$$
\underline{Claim 1:} 
$$
w(n_1,\dots,n_k) = \left(\frac{z}{1-\alpha z}\right)^k \cdot \prod_{j=1}^k   \mu\bigl(t^{n_j-n_{j-1}}\bigr).
$$
\underline{Proof of Claim 1:} For $k=1$, we decompose all paths by the intermediate steps within $G_0$ until time $s(1)$ and set $x_0:=e$, $n_0:=0$:
\begin{eqnarray*}
w(n_1) &=& \sum_{m\geq 1} \underbrace{\sum_{g_1,\dots,g_{m-1}\in G_0} \mathbb{P}[\forall j\in\{1,\dots,m-1\}: X_{j}=g_j]}_{=\alpha^{m-1}}\cdot z^{m-1}\cdot  \mu(t^{n_1}) \cdot z \\
&=& \frac{z}{1-\alpha z}\cdot \mu(t^{n_1}) =\frac{z}{1-\alpha z}\cdot  \mu(t^{n_1-n_0}).
\end{eqnarray*}
We remark that, for all $m\in\N$ and $h\in G_0$, we have the following equation due to group invariance of our  random walk on $G$:
\begin{eqnarray*}
&&\sum_{g_1,\dots,g_{m-1}\in G_0} \mathbb{P}_{t^{k-1}}[\forall j\in\{1,\dots,m-1\}: X_{j}=t^{k-1}g_j]\\
&=& \sum_{g_1,\dots,g_{m-1}\in G_0} \mathbb{P}_{t^{k-1}h}[\forall j\in\{1,\dots,m-1\}: X_{j}=t^{k-1}hg_j].
\end{eqnarray*}
Now we can conclude analogously by induction:
\begin{eqnarray*}
&&w(n_1,\dots,n_k) \\
&=& w(n_1,\dots,n_{k-1}) \\
&&\quad \cdot \sum_{\substack{m\geq 1, \\ g_1,\dots,g_{m-1}\in G_0}} \mathbb{P}_{t^{n_{k-1}}}[\forall j\in\{1,\dots,m-1\}: X_{j}=t^{n_{k-1}}g_j]\cdot  \mu\bigl(t^{n_k-n_{k-1}}\bigr)\cdot z^m  \\
&=& \left(\frac{z}{1-\alpha z}\right)^{k-1} \cdot \prod_{j=1}^{k-1} \mu\bigl(t^{n_{j}-n_{j-1}}\bigr) \cdot \frac{z}{1-\alpha z} \cdot \mu\bigl(t^{n_{k}-n_{k-1}}\bigr)\\
&=& \left(\frac{z}{1-\alpha z}\right)^k \cdot \prod_{j=1}^k \mu\bigl(t^{n_j-n_{j-1}}\bigr).
\end{eqnarray*}
This finishes the proof of Claim 1.
\par
Now we connect the random walk on $\mathbb{Z}$ with the random walk $(X_n)_{n\in\N_0}$ on $G$, for which we introduce the notation 
$$
G(e,A|z):=\sum_{n\geq 0}\mathbb{P}[X_n\in A]\,z^n=\sum_{g_0\in G_0} G(e,g_0|z).
$$ 
\underline{Claim 2:}
$$
G(e,A|z)=G_{\mathbb{Z}}\left(\frac{(1-\alpha)z}{1-\alpha z}\right)\cdot \frac{1}{1-\alpha z}.
$$
\underline{Proof of  Claim 2:} First, we recall that $A=G_0$ and observe the following equation:
$$
\sum_{n\geq 0}\mathbb{P}[X_n\in A,s(1)>n]\,z^n= \frac{1}{1-\alpha z}.
$$
Furthermore, we recall that $\mu(t)=(1-\alpha)p$ and $\mu(t^{-1})=(1-\alpha)(1-p)$. By decomposing each path from $e$ to $A$ by the number $k$ of transitions from the sets $t^{m}G_0$ to $t^{m\pm 1}G_0$, we obtain:
\begin{eqnarray*}
&&G(e,A|z)\\
&=&  \frac{1}{1-\alpha z}+ \sum_{k\geq 1} \sum_{n_1,\dots,n_{k-1}\in \mathbb{Z}} w(n_1,\dots,n_{k-1},0)  \cdot \frac{1}{1-\alpha z}\\
&=&  \frac{1}{1-\alpha z}+ \sum_{k\geq 1}  \left(\frac{z}{1-\alpha z}\right)^k \cdot \sum_{n_1,\dots,n_{k-1}\in\mathbb{Z}}\prod_{j=1}^{k-1} \mu\bigl(t^{n_j-n_{j-1}}\bigr)
 \cdot \mu(t^{-n_{k-1}})  \cdot \frac{1}{1-\alpha z}\\
&=&  \frac{1}{1-\alpha z} \cdot \biggl[1+\sum_{k\geq 1}  \left(\frac{(1-\alpha)z}{1-\alpha z}\right)^k \cdot \underbrace{\sum_{n_1,\dots,n_{k-1}\in\mathbb{Z}}\prod_{j=1}^{k-1} \mu_{\mathbb{Z}}(n_j-n_{j-1})\mu_{\mathbb{Z}}(-n_{k-1})}_{=\mu_{\mathbb{Z}}^{(k)}(0)}\biggr]\\
&\leq &   G_{\mathbb{Z}}\left( \frac{(1-\alpha)z}{1-\alpha z}\right)\cdot \frac{1}{1-\alpha z}.
\end{eqnarray*}
This finishes the proof of Claim 2.
\par
Since $G(e,A|z)\geq G(e,e|z)$ the lemma follows now from Claim 2 and the fact that $G_{\mathbb{Z}}(z)$ has radius of convergence strictly bigger than $1$.
\end{proof}

In the following we give the proof of Lemmas \ref{lem:tau-expmom} and \ref{lem:T-expmom} in the remaining case:
\begin{lemma} \label{lem:tau-expmom-A=G}
Consider the case $A=B=G_0$ with $p\in(0,1), p\neq \frac12$. Then the random variables $\tau_1-\tau_0$, $T_0$ and $T_1-T_0$ have exponential moments. 
\end{lemma}
\begin{proof}
If $A=B=G_0$ and $p\neq \frac12$, then $\W_k$ has the form $e_0t^kb_k$, $b_k\in B$, for all $k\in\N$, if $p>\frac12$, and  $e_0t^{-1}a_k$, $a_k\in A$, for all $k\in\N$, if $p<\frac12$: this is an easy consequence of transience of the projected random walk $\bigl(\psi(X_n)\bigr)_{n\in\N_0}$ onto $\mathbb{Z}$ from Lemma \ref{lem:recurrence}. We show again that  $(e_0te_0,1)$ can be reached from any other state of $(\W_k,\i{k})_{k\in\N}$ in two steps, where we restrict ourselves to the case $p>\frac12$ (the case $p<\frac12$ works analogously). For $(e_0tb,n_0)\in\mathbb{D}$, choose $n_b\in\N$ with $\mu_0^{(n_b)}(b^{-1})>0$; then
\begin{eqnarray*}
q\bigl((e_0tb,n_0),(e_0te_0,n_b+1) \bigr)&\geq& \frac{\xi(te_0)}{\xi(t_0b)}\cdot \alpha^{n_b}\cdot \mu_0^{(n_b)}(b^{-1})\cdot (1-\alpha)\cdot p>0\ \textrm{ and }\\
q\bigl((e_0te_0,n_b+1),(e_0te_0,1)\bigr) &\geq & (1- \alpha)\cdot p>0,
\end{eqnarray*}
which provides
$$
q:=\min_{b\in B} q\bigl((e_0tb,n_0),(e_0te_0,n_b+1)\bigr)\cdot q\bigl((e_0te_0,n_b+1),(e_0te_0,1)\bigr)>0.
$$
%Since there are infinitely many such $n_1\in\N$ with $\mu_0^{(n_1)}(b_0^{-1})>0$ due to $\mathrm{supp}(\mu_0)=G_0$, we get $q<1$. 
This leads to the desired exponential decay:
$$
\mathbb{P}[\tau_1-\tau_0=n]\leq (1-q)^{\lfloor \frac{n}{2}\rfloor},
$$
that is, $\tau_1-\tau_0$ has exponential moments.
\par
Existence of exponential moments of $T_1-T_0$ follows analogously as in \mbox{Lemma \ref{lem:T-expmom}:} after time $T_0$ the random walk $(X_n)_{n\in\N_0}$ can produce the next regeneration time $T_{1}$ in at most 
$$
N:=\max\{n_h\mid h\in A\cup B\}+2
$$ 
steps, where $n_h:=\min\{ m\in\N\mid \mu_0^{(m)}(h^{-1})\}$. Hence,  there is some $q_T\in (0,1)$ such that 
$$
\mathbb{P}[T_1-T_0=n]\leq (1-q_T)^{\lfloor \frac{n}{N}\rfloor},
$$
which yields existence of exponential moments of $T_1-T_0$. The same reasoning shows existence of exponential moments of $T_0$, which finishes the proof.
\end{proof}

The following lemma is left from Section \ref{sec:clt}, where we introduced the sequence of random variables $(L_i)_{i\in\N}$ in (\ref{def:Li}). 
\begin{lemma}\label{lem:Li-iid}
 $(L_i)_{i\in\N}$  forms an i.i.d. sequence of random variables.
\end{lemma}
\begin{proof}
Let be $i\in\N$, $z\in \mathbb{R}$. For $x_0\in G$ with $\P[X_{\e{\tau_i}}=x_0]>0$ and $m\in\N$, denote by $\mathcal{P}^{(1)}_{i,x_0,m}$ the set of paths $(e,w_1,\dots,w_m=x_0)\in G^{m+1}$ (with $\mu(w_{i-1}^{-1}w_i)>0$) of length $m$ such that \mbox{$[X_1=w_1,\dots,X_m=w_m]\cap [X_m=x_0,\e{\tau_i}=m]\neq \emptyset$.} Furthermore, denote by $\mathcal{P}^{(2)}_{i,x_0,m,n,z}$ the set of paths $(x_0,y_1,\dots,y_n)\in G^{n+1}$ of length $n\in\N$ such that 
$$
[X_m=x_0,X_{m+1}=y_1,\dots,X_{m+n}=y_n]\cap \left[\begin{array}{c}X_m=x_0,\e{\tau_{i-1}}=m,\\ \e{\tau_{i}}=m+n,L_i=z\end{array}\right]\neq \emptyset.
$$
By decomposing all paths until time $\e{\tau_{i}}$ into the part until time $\e{\tau_{i-1}}$ and into the part between times $\e{\tau_{i-1}}$ and $\e{\tau_{i}}$  we obtain:
\begin{eqnarray*}
&&\P[L_i=z] \\
&=& \sum_{\substack{x_0\in G:\\ \P[X_{\e{\tau_{i-1}}}=x_0]>0}} \P\bigl[ X_{\e{\tau_{i-1}}}=x_0,L_i=z\bigr]\\
&=&  \sum_{\substack{x_0\in G:\\ \P[X_{\e{\tau_{i-1}}}=x_0]>0}} \sum_{m\geq 1} \sum_{(e,w_1,\dots,w_m)\in \mathcal{P}^{(1)}_{i-1,x_0,m}} \P\bigl[X_1=w_1,\dots,X_m=w_m\bigr] \\
&&\quad \cdot \sum_{n\geq 1} \sum_{(x_0,y_1,\dots,y_n)\in \mathcal{P}^{(2)}_{i,x_0,m,n,z}} \P_{x_0}\bigl[X_1=y_1,\dots,X_n=y_n\bigr] \\
&&\quad \cdot \P_{y_n}\bigl[\forall l\geq 1: X_l \textrm{ has prefix } [y_n]\bigr]\\
&=&\underbrace{\sum_{\substack{x_0\in G:\\ \P[X_{\e{\tau_{i-1}}}=x_0]>0}} \sum_{m\geq 1} \sum_{(e,w_1,\dots,w_m)\in \mathcal{P}^{(1)}_{i-1,x_0,m}} \P\bigl[X_1=w_1,\dots,X_m=w_m\bigr]}_{=\bigl(1-\xi(te_0)\bigr)^{-1}} \\
&&\quad \cdot \sum_{\substack{n\geq 1,\\ (x_0,y_1,\dots,y_n)\in \mathcal{P}^{(2)}_{i,x_0,m,n,z}}} \P_{t}\bigl[X_1=tx_0^{-1}y_1,\dots,X_n=tx_0^{-1}y_n\bigr] \cdot \bigl(1-\xi(te_0)\bigr).
\end{eqnarray*}
In the last equation we used group invariance of our underlying random walk.
Observe  that paths $(x_0,y_1,\dots,y_n)\in \mathcal{P}^{(2)}_{i,x_0,m,n,z}$ lie completely in the set of words having \mbox{prefix $[x_0]$.} Therefore, there is a 1-to-1 correspondence between paths in $\mathcal{P}^{(2)}_{i,x_0,m,n,z}$ and $\mathcal{P}^{(2)}_{1,t,1,n,z}$, which lies completely in the set of words having prefix $t$, established by the shift $g \mapsto tx_0^{-1}g$. Therefore,
$$
\P[L_i=z] = \sum_{n\geq 1} \sum_{(t,y_1,\dots,y_n)\in \mathcal{P}^{(2)}_{1,t,1,n,z}} \P_t\bigl[X_1=y_1,\dots,X_n=y_n\bigr] .
$$
This proves that the $L_i$'s have the same distribution. An analogous decomposition of all possible paths proves independence, which we leave as an exercise to the interested reader.
\end{proof}

\end{appendix}

\section*{Acknowledgements}
The author is grateful to the anonymous referee for several hints regarding content and exposition.

\bibliographystyle{alea3}
\bibliography{literatur}

\begin{thebibliography}{43}
\providecommand{\natexlab}[1]{#1}
\providecommand{\url}[1]{\texttt{#1}}
\providecommand{\urlprefix}{URL }
\expandafter\ifx\csname urlstyle\endcsname\relax
  \providecommand{\doi}[1]{doi:\discretionary{}{}{}#1}\else
  \providecommand{\doi}{doi:\discretionary{}{}{}\begingroup
  \urlstyle{rm}\Url}\fi
\providecommand{\eprint}[2][]{\url{#2}}

\bibitem[{Bellman(1954)}]{bellman}
R.~Bellman.
\newblock Limit theorems for non-commutative operations. {I}.
\newblock \emph{Duke Math. J.} \textbf{21}, 491--500 (1954).

\bibitem[{Benjamini and Peres(1994)}]{benjamini-peres94}
I.~Benjamini and Y.~Peres.
\newblock Tree-indexed random walks on groups and first passage percolation.
\newblock \emph{Probab. Theory Rel. Fields} \textbf{98}~(1), 91--112 (1994).

\bibitem[{Billingsley(1999)}]{billingsley:99}
P.~Billingsley.
\newblock \emph{Convergence of Probability Measures}.
\newblock Wiley (1999).

\bibitem[{Bj\"orklund(2010)}]{bjoerklund10}
M.~Bj\"orklund.
\newblock Central limit theorems for {G}romov hyperbolic groups.
\newblock \emph{J. Theoret. Probab.} \textbf{23}~(3), 871--887 (2010).

\bibitem[{Blach\`ere et~al.(2008)Blach\`ere, Ha\"issinsky and
  Mathieu}]{blachere-haissinsky-mathieu}
S.~Blach\`ere, P.~Ha\"issinsky and P.~Mathieu.
\newblock Asymptotic entropy and {G}reen speed for random walks on countable
  groups.
\newblock \emph{Ann. Probab.} \textbf{36}~(3), 1134--1152 (2008).

\bibitem[{Candellero and Gilch(2009)}]{candellero-gilch}
E.~Candellero and L.~Gilch.
\newblock Phase transitions for random walk asymptotics on free products of
  groups.
\newblock \emph{Random Structures \& Algorithms} \textbf{40}~(2), 150--181
  (2009).

\bibitem[{Cartwright and Soardi(1986)}]{cartwright-soardi}
D.I. Cartwright and P.M. Soardi.
\newblock Random walks on free products, quotients and amalgams.
\newblock \emph{Nagoya Math. J.} \textbf{102}, 163--180 (1986).

\bibitem[{Cuno and Sava-Huss(2018)}]{cuno-sava-huss}
J.~Cuno and E.~Sava-Huss.
\newblock Random walks on {B}aumslag--{S}olitar groups.
\newblock \emph{Israel Journal of Mathematics} \textbf{228}~(2), 627--663
  (2018).

\bibitem[{Derriennic(1980)}]{derriennic}
Y.~Derriennic.
\newblock Quelques applications du th\'eor\`eme ergodique sous-additif.
\newblock \emph{Ast\'erisque} \textbf{74}, 183--201 (1980).

\bibitem[{Dyubina(1999)}]{erschler99}
A.~G. Dyubina.
\newblock An example of the rate of departure to infinity for a random walk on
  a group.
\newblock \emph{Uspekhi Mat. Nauk} \textbf{54}~(5(329)), 159--160 (1999).

\bibitem[{Erschler(2001)}]{erschler01}
A.~Erschler.
\newblock Asymptotics of drift and entropy for a random walk on groups.
\newblock \emph{Uspekhi Mat. Nauk} \textbf{56}~(3(339)), 179--180 (2001).

\bibitem[{Furstenberg and Kesten(1960)}]{furstenberg-kesten60}
H.~Furstenberg and H.~Kesten.
\newblock Products of random matrices.
\newblock \emph{Ann. Math. Statist.} \textbf{31}~(2), 457--469 (1960).

\bibitem[{Gairat et~al.(1995)Gairat, Malyshev, Menshikov and
  Pelikh}]{malyshev2}
A.~Gairat, V.~Malyshev, M.~Menshikov and K.~Pelikh.
\newblock Classification of {M}arkov chains describing the evolution of random
  strings.
\newblock \emph{Russian Math. Surveys} \textbf{50}~(2), 237--255 (1995).

\bibitem[{Gerl and Woess(1986)}]{gerl-woess}
P.~Gerl and W.~Woess.
\newblock Local limits and harmonic functions for nonisotropic random walks on
  free groups.
\newblock \emph{Probab. Theory Rel. Fields} \textbf{71}, 341--355 (1986).

\bibitem[{Gilch(2007)}]{gilch:07}
L.~Gilch.
\newblock Rate of escape of random walks on free products.
\newblock \emph{J. Aust. Math. Soc.} \textbf{83}~(I), 31--54 (2007).

\bibitem[{Gilch(2008)}]{gilch:08}
L.~Gilch.
\newblock Rate of escape of random walks on regular languages and free products
  by amalgamation of finite groups.
\newblock \emph{Discrete Math. Theor. Comput. Sci. Proc.} pages 405--420
  (2008).

\bibitem[{Gilch(2011)}]{gilch:11}
L.~Gilch.
\newblock Asymptotic entropy of random walks on free products.
\newblock \emph{Electron. J. Probab.} \textbf{16}, 76--105 (2011).

\bibitem[{Gilch(2016)}]{gilch:15}
L.~Gilch.
\newblock Asymptotic entropy of random walks on regular languages over a finite
  alphabet.
\newblock \emph{Electron. J. Probab.} \textbf{21}~(8), 1--42 (2016).

\bibitem[{Gilch and Ledrappier(2013)}]{gilch-ledrappier13}
L.~Gilch and F.~Ledrappier.
\newblock Regularity of the drift and entropy of random walks on groups.
\newblock \emph{Publ. Mat. Urug.} \textbf{14}, 147--158 (2013).

\bibitem[{Gou\"ezel(2017)}]{gouezel:15}
S.~Gou\"ezel.
\newblock Analyticity of the entropy and the escape rate of random walks in
  hyperbolic groups.
\newblock \emph{Discrete Analysis} \textbf{7}, 1--37 (2017).

\bibitem[{Guivarc'h(1980)}]{guivarch}
Y.~Guivarc'h.
\newblock Sur la loi des grands nombres et le rayon spectral d'une marche
  al\'eatoire.
\newblock \emph{Ast\'erisque} \textbf{74}, 47--98 (1980).

\bibitem[{Ha\"issinsky et~al.(2018)Ha\"issinsky, Mathieu and
  M\"uller}]{haissinsky-mathieu-mueller12}
P.~Ha\"issinsky, P.~Mathieu and S.~M\"uller.
\newblock Renewal theory for random walks on surface groups.
\newblock \emph{Ergodic Theory Dynam. Systems} \textbf{38}~(1), 155--179
  (2018).

\bibitem[{Higman et~al.(1949)Higman, Neumann and Neumann}]{hnn:49}
G.~Higman, B.~Neumann and H.~Neumann.
\newblock Embedding theorems for groups.
\newblock \emph{J. London Math. Soc.} \textbf{s1--24}~(4), 247--254 (1949).

\bibitem[{Kaimanovich(1991)}]{kaimanovich:91}
V.~Kaimanovich.
\newblock Poisson boundaries of random walks on discrete solvable groups.
\newblock \emph{Probability measures on groups X (Oberwolfach, 1990), Plenum,
  New York} pages 205--238 (1991).

\bibitem[{Kingman(1968)}]{kingman}
J.F.C. Kingman.
\newblock The ergodic theory of subadditive processes.
\newblock \emph{J. Roy. Statist. Soc., Ser. B} \textbf{30}, 499--510 (1968).

\bibitem[{Lalley(1993)}]{lalley93}
S.~Lalley.
\newblock Finite range random walk on free groups and homogeneous trees.
\newblock \emph{Ann. Probab.} \textbf{21}~(4), 2087--2130 (1993).

\bibitem[{Lalley(2000)}]{lalley}
S.~Lalley.
\newblock Random walks on regular languages and algebraic systems of generating
  functions.
\newblock \emph{Algebraic Methods in Statistics and Probability, Contemp.
  Math.} \textbf{287}~(201--230) (2000).

\bibitem[{Ledrappier(2001)}]{ledrappier01}
F.~Ledrappier.
\newblock Some asymptotic properties of random walks on free groups.
\newblock \emph{CRM Proceedings and Lectures Notes} \textbf{21}, 117--152
  (2001).

\bibitem[{Ledrappier(2012)}]{ledrappier12}
F.~Ledrappier.
\newblock Analyticity of the entropy for some random walks.
\newblock \emph{Groups Geom. Dyn.} \textbf{6}, 317--333 (2012).

\bibitem[{Ledrappier(2013)}]{ledrappier12-2}
F.~Ledrappier.
\newblock Regularity of the entropy for random walks on hyperbolic groups.
\newblock \emph{Ann. Probab.} \textbf{41}~(5), 3582--3605 (2013).

\bibitem[{Lyndon and Schupp(1977)}]{lyndon-schupp}
R.~Lyndon and P.~Schupp.
\newblock \emph{Combinatorial Group Theory}.
\newblock Springer-Verlag (1977).

\bibitem[{Mairesse and Math\'eus(2007{\natexlab{a}})}]{mairesse1}
J.~Mairesse and F.~Math\'eus.
\newblock Random walks on free products of cyclic groups.
\newblock \emph{J. London Math. Soc.} \textbf{75}~(1), 47--66
  (2007{\natexlab{a}}).

\bibitem[{Mairesse and Math\'eus(2007{\natexlab{b}})}]{mairesse-matheus:2007}
J.~Mairesse and F.~Math\'eus.
\newblock Randomly growing braid on three strands and the manta ray.
\newblock \emph{Ann. Applied Proba.} \textbf{17}, 502--536
  (2007{\natexlab{b}}).

\bibitem[{Malyshev(1995)}]{malyshev}
V.~Malyshev.
\newblock Stabilization laws in the evolution of a random string.
\newblock \emph{Problems Inform. Transmission} \textbf{30}, 260--274 (1995).

\bibitem[{Malyshev(1996)}]{malyshev-inria}
V.~Malyshev.
\newblock Interacting strings of characters.
\newblock Technical Report 3057, INRIA (1996).

\bibitem[{Mathieu(2015)}]{mathieu:15}
P.~Mathieu.
\newblock Differentiating the entropy of random walks on hyperbolic groups.
\newblock \emph{Ann. Probab.} \textbf{43}~(1), 166--187 (2015).

\bibitem[{Nagnibeda and Woess(2002)}]{woess2}
T.~Nagnibeda and W.~Woess.
\newblock Random walks on trees with finitely many cone types.
\newblock \emph{J. Theoret. Probab.} \textbf{15}, 399--438 (2002).

\bibitem[{Sawyer(1978)}]{sawyer78}
S.~Sawyer.
\newblock Isotropic random walks in a tree.
\newblock \emph{Zeitschrift f. Wahrscheinlichkeitstheorie} \textbf{Verw. Geb.
  42}, 279--292 (1978).

\bibitem[{Sawyer and Steger(1987)}]{sawyer}
S.~Sawyer and T.~Steger.
\newblock The rate of escape for anisotropic random walks in a tree.
\newblock \emph{Probab. Theory Rel. Fields} \textbf{76}, 207--230 (1987).

\bibitem[{Stallings(1971)}]{stallings:71}
J.~Stallings.
\newblock \emph{Group theory and three-dimensional manifolds}.
\newblock Yale Mathematical Monographs, Yale University Press, New Haven,
  Conn.-London, a {J}ames {K}. {W}hittemore lecture in mathematics given at
  {Y}ale {U}niversity, 1969 edition (1971).

\bibitem[{Woess(1986)}]{woess3}
W.~Woess.
\newblock Nearest neighbour random walks on free products of discrete groups.
\newblock \emph{Boll. Un. Mat. Ital.} \textbf{5-B}, 961--982 (1986).

\bibitem[{Woess(1989)}]{woess-isr}
W.~Woess.
\newblock Boundaries of random walks on graphs and groups with infinitely many
  ends.
\newblock \emph{Israel Journal of Mathematics} \textbf{68}~(3), 271--301
  (1989).

\bibitem[{Woess(2000)}]{woess}
W.~Woess.
\newblock \emph{Random Walks on Infinite Graphs and Groups}.
\newblock Cambridge University Press (2000).

\end{thebibliography}

\end{document}